\documentclass[12pt,a4paper]{amsart}

\usepackage[all]{xy}
\usepackage{amssymb,amsfonts,mathtools}
\usepackage{bm} 
\usepackage{etex}
\usepackage{graphicx}
\usepackage{mathrsfs}
\usepackage{fullpage}

\theoremstyle{plain}
\newtheorem{theorem}{Theorem}
\newtheorem{lemma}[theorem]{Lemma}
\newtheorem{proposition}[theorem]{Proposition}

\newtheorem{definition}[theorem]{Definition}

\theoremstyle{definition}

\theoremstyle{remark}

\newtheorem{remark}{Remark}

\makeatletter
\newif\if@borderstar
\def\bordermatrix{\@ifnextchar*{%
\@borderstartrue\@bordermatrix@i}{\@borderstarfalse\@bordermatrix@i*}%
}
\def\@bordermatrix@i*{\@ifnextchar[{\@bordermatrix@ii}{\@bordermatrix@ii[() ]}}
\def\@bordermatrix@ii[#1]#2{%
\begingroup
\m@th\@tempdima8.75\p@\setbox\z@\vbox{%
\def\cr{\crcr\noalign{\kern 2\p@\global\let\cr\endline }}%
\ialign {$##$\hfil\kern 2\p@\kern\@tempdima & \thinspace %
\hfil $##$\hfil && \quad\hfil $##$\hfil\crcr\omit\strut %
\hfil\crcr\noalign{\kern -\baselineskip}#2\crcr\omit %
\strut\cr}}%
\setbox\tw@\vbox{\unvcopy\z@\global\setbox\@ne\lastbox}%
\setbox\tw@\hbox{\unhbox\@ne\unskip\global\setbox\@ne\lastbox}%
\setbox\tw@\hbox{%
$\kern\wd\@ne\kern -\@tempdima\left\@firstoftwo#1%
\if@borderstar\kern2pt\else\kern -\wd\@ne\fi%
\global\setbox\@ne\vbox{\box\@ne\if@borderstar\else\kern 2\p@\fi}%
\vcenter{\if@borderstar\else\kern -\ht\@ne\fi%
\unvbox\z@\kern-\if@borderstar2\fi\baselineskip}%
\if@borderstar\kern-2\@tempdima\kern2\p@\else\,\fi\right\@secondoftwo#1 $%
}\null \;\vbox{\kern\ht\@ne\box\tw@}%
\endgroup}
\makeatother 

\DeclareMathOperator{\Cov}{Cov}
\DeclareMathOperator{\Expectation}{\mathbb E} 
\DeclareMathOperator{\Hessian}{Hess}

\DeclareMathOperator{\Span}{Span}

\DeclareMathOperator{\arsinh}{arsinh}
\DeclareMathOperator{\entropyprod}{\mathscr{D}}

\newcommand{\eTof}[1]{T\maxexp{#1}}
\newcommand{\mTof}[1]{\prescript{*}{}T\maxexp{#1}}

\newcommand{\entropyprodat}[1]{\entropyprod\left(#1\right)}
\newcommand{\Bspace}[1]{B_{#1}}
\newcommand{\KH}[2]{\operatorname{DH}\left(#1 \middle| #2 \right)}
\newcommand{\KL}[2]{\mathbf{D}\left(#1\,\Vert#2\right)}
\newcommand{\Lexp}[1]{L^{\cosh-1}\left(#1\right)}
\newcommand{\LlogL}[1]{L^{(\cosh-1)_*}\left(#1\right)}
\newcommand{\Maxexp}{\mathcal E}
\newcommand{\PSexp}[2]{\euler^{#2 - K_{#1}(#2)}}
\newcommand{\Sspace}[1]{\mathcal S_{#1}}
\newcommand{\absoluteval}[1]{\left|#1\right|}
\newcommand{\bDelta}{{\bm\Delta}}
\newcommand{\bdelta}{{\bm\delta}}
\newcommand{\bnabla}{{\bm\nabla}}
\newcommand{\collinv}[2]{\absoluteval{#1}^2+\absoluteval{#2}^2}
\newcommand{\condexpat}[3]{\Expectation_{#1}\left(#2 \middle\vert #3 \right)}
\newcommand{\condexp}[2]{\Expectation\left(#1 \middle\vert #2 \right)}
\newcommand{\cosof}[1]{\cos\left(#1\right)}
\newcommand{\covat}[3]{\Cov_{#1}\left(#2,#3\right)}

\newcommand{\densities}{\mathcal P_{\ge}}
\newcommand{\derivby}[1]{\frac{d}{d#1}}

\newcommand{\eBspace}[1]{B_{#1}}

\newcommand{\etransport}[2]{\prescript{\text{e}}{} {\mathbb U} _ {#1} ^ {#2}}
\newcommand{\euler}{\mathrm{e}}
\newcommand{\expectat}[2]{{\Expectation}_{#1}\left[#2\right]}
\newcommand{\expectof}[1]{\Expectation\left(#1\right)}

\newcommand{\expof}[1]{\exp\left(#1\right)}
\newcommand{\hessianof}[1]{\Hessian #1}

\newcommand{\integrald}[3]{\int_{#1} {#2} \ {#3}}

\newcommand{\logof}[1]{\log\left(#1\right)}
\newcommand{\mBspace}[1]{\prescript{*}{}{B}_{#1}}
\newcommand{\maxexpM}{{\mathcal E}}
\newcommand{\maxexp}[1]{{\mathcal E}\left(#1\right)}
\newcommand{\maxexpat}[2]{{\mathcal E}_{#1}\left(#2\right)}

\newcommand{\mtransport}[2]{\prescript{\text{m}}{} {\mathbb U} _ {#1} ^ {#2}}

\newcommand{\normat}[2]{\left\Vert#2\right\Vert_{#1}}

\newcommand{\pdensities}{\mathcal P_>}
\newcommand{\preBspace}[1]{\prescript{*}{}B_{#1}}
\newcommand{\reals}{\mathbb R}
\newcommand{\scalarat}[3]{\left\langle#2,#3\right\rangle_{#1}}
\newcommand{\scalarof}[2]{\left\langle#1,#2\right\rangle}
\newcommand{\sdensities}{\mathcal P_{1}}
\newcommand{\sdomain}[1]{\mathcal S_{#1}}
\newcommand{\setof}[2]{\left\{#1 \middle| #2 \right\}}
\newcommand{\set}[1]{\left\{#1\right\}}
\newcommand{\sinof}[1]{\sin\left(#1\right)}
\newcommand{\spanof}[1]{\Span\left(#1\right)}

\newcommand{\transport}[2]{\mathbb U_{#1}^{#2}}
\newcommand{\versof}[1]{\widehat{#1}}

\newcommand{\widebar}[1]{\overline #1}
\renewcommand{\ge}{\geqslant}
\renewcommand{\le}{\leqslant}
\renewcommand{\geq}{\geqslant}
\renewcommand{\leq}{\leqslant}
\newcommand{\Sd}{\mathbb{S}^{2}}

\begin{document}

\title[IG and Boltzmann]{Information Geometry Formalism for the Spatially Homogeneous Boltzmann Equation}

\author[B. Lods]{Bertrand Lods}

\address{Universit\`a di Torino \& Collegio Carlo Alberto,  Department of Economics and Statistics, Corso Unione Sovietica, 218/bis, 10134 Torino, Italy}

\email{bertrand.lods@unito.it}

\author[G. Pistone]{Giovanni Pistone}

\address{de Castro Statistics, Collegio Carlo Alberto, Via Real Collegio 30, 10024 Moncalieri, Italy}

\email{giovanni.pistone@carloalberto.org}

\thanks{Both B. Lods and G. Pistone acknowledge support from the \emph{de Castro Statistics Initiative}, Collegio Carlo Alberto, Moncalieri, Italy. G. Pistone is a member of INdAM/GNAMPA.} 


\keywords{Information Geometry, Orlicz Space, Spatially Homogeneous Boltzmann Equation, Kullback-Leibler divergence, Hyv\"arinen divergence}

\date{\today}

\begin{abstract}
Information Geometry generalizes to infinite dimension by modeling the tangent space of the relevant manifold of probability densities with exponential Orlicz spaces. We review here several properties of the  exponential manifold on a suitable set $\mathcal E$ of mutually absolutely continuous densities. We study in particular the fine properties of the Kullback-Liebler divergence in this context. We also show that this setting is well-suited for the study of the spatially homogeneous Boltzmann equation  if $\mathcal E$ is a set of positive densities with finite relative entropy with respect to the Maxwell density. More precisely, we analyse the Boltzmann operator in the geometric setting from the point of its Maxwell's weak form as a composition of elementary operations in the exponential manifold, namely tensor product, conditioning, marginalization and we prove in a geometric way the basic facts i.e., the H-theorem. We also illustrate the robustness of our method by discussing, besides the Kullback-Leibler divergence, also the property of Hyv\"arinen divergence. This requires to generalise our approach to Orlicz-Sobolev spaces to include derivatives.%
\end{abstract}

\maketitle

\tableofcontents

\section{Introduction}

Information geometry (IG) has been essentially developed by S.-I. Amari, see the monograph by Amari and Nagaoka \cite{MR1800071}. In his work, all previous geometric---essentially metric---descriptions of probabilistic and statistics concepts are extended in the direction of affine differential geometry, including the fundamental treatment of connections. A corresponding concept for abstract manifold, called statistical manifold, has been worked out by S. L. Lauritzen in \cite{MR932246}. Amari's framework is today considered a case of Hessian geometry as it is described in the monograph by H. Shima \cite{MR2293045}.

Other versions of IG have been studied to deal with a non-parametric settings such as the Boltzmann equation as it is described in \cite{MR1313028} and \cite{MR1942465}. A very general set-up for information geometry is the following. Consider a one-dimensional family of positive densities with respect to a measure $\mu$, $\theta \mapsto p_\theta$, and a random variable $U$. A classical statistical computation, possibly due to Ronald Fisher, is

\begin{multline*}
  \derivby \theta \int U(x) p(x;\theta)\ \mu(dx) =   \int U(x) \left(\derivby \theta p(x;\theta)\right)\ \mu(dx) = \\  \int U(x) \left(\derivby \theta \log p(x;\theta)\right)\ p(x;\theta)\ \mu(dx) = \expectat \theta {\left(U - \expectat \theta U\right) \derivby \theta \log p_\theta } \ .
\end{multline*}
The previous computation suggests the following geometric construction which is rigorous if the sample space is finite and can be forced to work in general under suitable assumptions. We use the differential geometry language e.g., \cite{MR1138207}. If $\Delta$ is the \emph{probability simplex} on a given sample space $(\Omega,\mathcal F)$, we define the \emph{statistical bundle} of $\Delta$ to be
 
\begin{equation*} 
 T\Delta = \setof{(\pi,u)}{\pi \in \Delta, u \in \mathcal L^2(\pi), \expectat \pi u = 0}.
\end{equation*}
Given a one dimensional curve in $\Delta$, $\theta \mapsto \pi_\theta$ we can define its \emph{velocity} to be the curve

\begin{equation*}
  \theta \mapsto \left(\pi_\theta,\frac D {d\theta} \pi_\theta\right) \in T\Delta \ ,
\end{equation*}
where we define

\begin{equation*}
  \frac D {d\theta} \pi_\theta = \frac d {d\theta} \logof{\pi_\theta} = \dfrac {\frac d {d\theta} \pi_\theta} {\pi_\theta} \ .
\end{equation*}
Each fiber $T_\pi\Delta = \mathcal L_0^2(\mu)$ has a scalar product and we have a parallel transport

\begin{equation*}
  \transport \pi \nu \colon T_\pi\Delta \ni V \mapsto V - \expectat \nu V \in T_\nu\Delta \ .
\end{equation*}

This structure provides an interesting framework to interpret the Fisher computation cited above. The basic case of a finite state space has been extended by Amari and coworkers to the case of a parametric set of strictly positive probability densities on a generic sample space. Following a suggestion by A. P.  Dawid in \cite{MR0428531,MR0471125}, a particular nonparametric version of that theory was developed in a series of papers \cite{MR1370295,MR1628177,MR1704564,MR1805840,cena:2002,MR2396032,imparato:thesis,pistone:2009EPJB,pistone:2010SL,MR3126029,MR3130268}, where the set $\pdensities$ of all strictly positive probability densities of a measure space is shown to be a Banach manifold (as it is defined in \cite{bourbaki:71variete,MR960687,MR1335233}) modeled on an Orlicz Banach space, see e.g., \cite[Chapter II]{MR724434}. 

In the present paper, Sec. \ref{sec:model-spaces} recalls the theory and our notation about the model Orlicz spaces. This material is included for convenience only and this part should be skipped by any reader aware of any of the papers \cite{MR1370295,MR1628177,MR1704564,MR1805840,cena:2002,MR2396032,imparato:thesis,pistone:2009EPJB,pistone:2010SL,MR3126029,MR3130268} quoted above. The following Sec. \ref{exponentialmanifold} is mostly based on the same references and it is intended to introduce that manifold structure and to give a first example of application to the study of Kullback-Liebler divergence. The special features of statistical manifolds that contain the Maxwell density are discussed in Sec. \ref{chap:gaussian-space}. In this case we can define the Boltzmann-Gibbs entropy and study its gradient flow. The setting for the Boltzmann equation is discussed in Sec. \ref{sec:boltzmann-1} where we show that the equation can be derived from probabilistic operations performed on the statistical manifold. Our application to the study of the Kullback-Liebler divergence is generalised in Sec. \ref{sec:orlicz-sobolev} to the more delicate case of the Hyvar\"inen divergence. This requires in particular a generalisation of the manifold structure to include differential operators and leads naturally to the introduction of Orlicz-Sobolev spaces.

We are aware that there are other approaches to non-parametric information geometry that are not based on the notion of the exponential family and that we do not consider here. We mention in particular \cite{MR2948226} and \cite{ay|jost|le|schwachhofer:2014}.

\section{Model Spaces}
\label{sec:model-spaces}

 Given a $\sigma$-finite measure space, $(\Omega, \mathcal F, \mu)$, we denote by $\pdensities$ the set of all densities that are positive $\mu$-a.s, by $\densities$ the set of all densities, by $\sdensities$ the set of measurable functions $f$ with $\int f\ d\mu = 1$. 
 
We introduce here the Orlicz spaces we shall mainly investigate in the sequel. We refer to \cite{MR2396032} and \cite[Chapter II]{MR724434} for more details on the matter. We consider the  Young function 

\begin{equation*}
\Phi \colon \reals \ni x \mapsto \Phi(x)=\cosh x - 1
\end{equation*}
and, for any $p \in \pdensities$, the Orlicz space $L^{\Phi}(p)=\Lexp p$ is defined as follows: a real random variable $U$ belongs to $L^{\Phi}(p)$ if 

\begin{equation*}
\expectat p {\Phi(\alpha U)} < +\infty \qquad \text{ for some } \quad \alpha > 0 \ .
\end{equation*}
The Orlicz space $L^{\Phi}(p)$ is a Banach space when endowed with the Luxemburg norm defined as

\begin{equation*}
\Vert U \Vert_{\Phi, p}=\inf \setof{\lambda > 0}{\,\expectat p{\Phi(U/\lambda)} \le 1} \ .
\end{equation*}
The conjugate function of $\Phi = \cosh - 1$ is 

\begin{equation*} 
  \Phi_{*}(y)=(\cosh - 1)_*(y) = \int_0^y \arsinh u\ du = y \arsinh y + 1 - \sqrt{1+y^2}, \qquad y \in \reals
\end{equation*}
which satisfies the so-called $\Delta_2$-condition as 

\begin{equation*}
\Phi_{*}(\alpha\,y) =\int_{0}^{\absoluteval{\alpha y}} \left(\absoluteval{\alpha y}-u\right)\Phi_{*}''(u)du= \int_0^{\absoluteval{\alpha y}} \frac{\absoluteval \alpha \absoluteval y - u}{1+u^2}\ du \le \max(1,\alpha^2) \Phi_{*}(y) \ .
\end{equation*}

Since $\Phi_{*}$ is a Young function, for any $p \in \pdensities$, one can define as above the associated Orlicz space $L^{\Phi_{*}}(p)=\LlogL p$ and its corresponding Luxemburg norm $\Vert \cdot \Vert_{\Phi_{*}, p}.$ Because the functions $\Phi$ and $\Phi_*$ form a Young pair, for each $U \in L^{\cosh-1}(p)$ and $V \in L^{(\cosh-1)_*}(p)$ we can deduce from Young's inequality 

$$xy \le \Phi(x) + \Phi_{*}(y) \qquad \forall x,y \in \reals^{2},$$
that the following Holder's inequality holds:
 
\begin{equation*}
\absoluteval{\expectat p {UV}} \le  2 \normat {\cosh-1,p} U\,\normat {(\cosh-1)_{*},p} V < \infty \ .
\end{equation*}
Moreover, it is a classical result that  the space $L^{(\cosh-1)_*}(p)$ is separable and its dual space is $L^{\cosh-1}(p)$, the duality pairing being 

\begin{equation*}
(V,U) \in  L^{(\cosh-1)_*}(p) \times L^{\cosh-1}(p)  \longmapsto \scalarat p U V := \expectat p {UV} \ .
\end{equation*}
We recall the following continuous embedding result that we shall use repeatedly in the paper:
\begin{theorem}\label{theo:embed}
Given $p \in \pdensities$, for any $1 < r < \infty$, the following  embeddings 

\begin{equation*}
L^{\infty}(p) \hookrightarrow \Lexp p \hookrightarrow L^{r}(p) \hookrightarrow \LlogL p \hookrightarrow L^{1}(p)
\end{equation*}
are continuous.
\end{theorem}
From this result, we deduce easily the following useful Lemma
\begin{lemma}\label{lem:technical}
Given $p \in \pdensities$ and $k \ge 1.$ For any $u_{1},\ldots, u_{k} \in \Lexp p$ one has 

\begin{equation*}
\prod_{i=1}^{k} u_{i} \in \bigcap_{1<r<\infty}L^{r}(p) \ .
\end{equation*}
\end{lemma}
\begin{proof}
According to Theorem \ref{theo:embed}, $u_{i} \in L^{r}(p)$ for any $1< r <\infty$ and any $i=1,\ldots,k$. The proof follows then simply from the repeated use of Holder inequality.
\end{proof}
From now on we also define, for any $p \in \pdensities$

\begin{equation*}
\eBspace p = L^\Phi_0(p) := \setof{u \in \Lexp p}{\expectat p u = 0} \ .
\end{equation*}
In the same way, we set 
$^{*}\eBspace p=L^{\Phi_{*}}_{0}(p):=\setof{u \in \LlogL p}{\expectat p u =0}.$

\subsection{Cumulant Generating Functional}

 Let $p \in \mathcal \pdensities$ be given. With the above notations one can define:
 \begin{definition}
 The \emph{cumulant generating functional} is the mapping

\begin{equation*}
K_p  \colon  u \in \eBspace p \longmapsto \log \expectat p {\euler^{u}}  \in [0,+\infty] \ .
\end{equation*} 
\end{definition}
The following result \cite{cena:2002,MR2396032} shows the properties of the exponential function as a superposition mapping \cite{MR1066204}.
\begin{proposition}\label{prop:expisanalytic}\ Let $a \geq 1$ and $p \in \pdensities$ be given.
\begin{enumerate}
\item
For any $n = 0, 1, \dots$ and $u \in L^\Phi(p)$:

\begin{equation*}
 \lambda_{a,n}(u) \colon \left(w_1, \dots, w_n \right)  \mapsto  \dfrac{w_1}{a} \cdots \dfrac{w_n}{a}\ \euler^{\frac ua}
\end{equation*}
is a continuous, symmetric, $n$-multi-linear map from $\left(L^\Phi(p)\right)^{n}$ to $L^{a}\left(p\right)$.%
\item
$v \mapsto \sum_{n = 0}^\infty \frac{1}{n!} \left(\dfrac{v}{a}\right)^n$ is a power series from $L^{\Phi}(p)$ to $L^a(p)$, with radius of convergence  larger than $1$.
\item The superposition mapping, $ v \mapsto \euler^{v/a}$, is an analytic function from the open unit ball of $L^\Phi(p)$ to $L^a(p)$.
\end{enumerate}
\end{proposition}
The cumulant generating functional enjoys the following properties (see \cite{MR1370295,cena:2002,MR2396032}):
\begin{proposition} 
\label{prop:cumulant}\ 
\begin{enumerate}
\item \label{prop:cumulant1}$K_p (0) = 0$; otherwise, for each $u \neq 0$, $K_p (u) > 0$.
\item \label{prop:cumulant2}$K_p$ is convex and lower semi-continuous, and its proper domain

\begin{equation*}
\mathrm{dom}(K_{p})=\setof{u \in \eBspace p}{K_{p}(u) < \infty} 
\end{equation*}
is a convex set that contains the open unit ball of $\eBspace p$.
\item \label{prop:cumulant3}$K_p$ is infinitely G\^ateaux-differentiable in the interior of its proper domain.
\item \label{prop:cumulant4}$K_p$ is bounded, infinitely Fr\'echet-differentiable and analytic on the open unit ball of $\eBspace p$.
\end{enumerate}
\end{proposition}
\begin{remark}
One sees from the above property 2 that the interior of the proper domain of $K_{p}$ is a non-empty open convex set. From now on we shall adopt the notation

\begin{equation*}
\sdomain p=\mathrm{Int}\left(\mathrm{dom}(K_{p})\right) \ .
\end{equation*}
\end{remark}

Other properties of the functional $K_p$ are described below, as they relate directly to the exponential manifold.

\section{Exponential manifold}
\label{exponentialmanifold}

 The set of positive densities, $\pdensities$, locally around a given $p \in \pdensities$, is modeled by the subspace of centered random variables in the Orlicz space, $L^{\cosh-1}(p)$. Hence, it is crucial to discuss the isomorphism of the model spaces for different $p$'s in order to show the existence of an atlas defining a Banach manifold.

\begin{definition}[\textit{\textbf{Statistical exponential manifold}}{ \cite[Def. 20]{MR2396032}}]
For $p \in \pdensities$, the \emph{statistical exponential manifold} at $p$ is

\begin{equation*}
  \maxexp p = \setof{\euler^{u - K_p(u)} p}{u \in \sdomain p} \ .
\end{equation*}
\end{definition}

We also need the following definition of connection

\begin{definition}[\textit{\textbf{Connected densities}}]\label{def:smile}
Densities $p,q \in \pdensities$ are connected by an open exponential arc if there exists an open exponential family containing both, {\em i.e.}, if for a neighborhood $I$ of $[0,1]$

\begin{equation*}
 \int_{\Omega} p^{1-t}q^t \ d\mu = \expectat p {\left(\frac qp\right)^t} = \expectat q {\left(\frac pq\right)^{1-t}} < +\infty, \quad t \in I
 \end{equation*}
 In such a case, one simply writes $p \smile q.$
 \end{definition}
\begin{theorem}[\textit{\textbf{Portmanteau theorem}} {\cite[Th 19 and 21]{MR2396032},\cite[Theorem 4.7]{santacroce|siri|trivellato:2015}}\label{prop:maxexp-pormanteau}]\ 
Let $p, q \in \pdensities.$ The following statements are equivalent:
\begin{enumerate}
\item \label{prop:maxexp-pormanteau-1} $p \smile q$ (i.e. $p$ and $q$ are connected by an open exponential arc);
  \item \label{prop:maxexp-pormanteau-2} $q \in \maxexp p$;
  \item \label{prop:maxexp-pormanteau-3} $\maxexp p = \maxexp q$;
  \item \label{prop:maxexp-pormanteau-4} $\log \frac q p \in L^{\Phi}(p) \cap L^{\Phi}(q)$;
  \item \label{prop:maxexp-pormanteau-5} $L^{\Phi}(p)=L^{\Phi}(q)$ (i.e. they both coincide  as vector spaces and their norms are equivalent);
  \item \label{prop:maxexp-pormanteau-6} There exists $\varepsilon >0$ such that 

\begin{equation*}
\frac{q}{p} \in L^{1+\varepsilon}(p) \quad \text{ and } \quad  \frac{p}{q} \in L^{1+\varepsilon}(q) \ 
\end{equation*}
\end{enumerate}
\end{theorem}

 We can now define the charts and atlas of the exponential manifold as follows:
 \begin{definition}[\textit{\textbf{Exponential manifold}} \cite{MR1370295,MR1704564,cena:2002,MR2396032}]
 For each $p \in \pdensities$, define the charts \emph{at $p$} as:

 \begin{equation*}
   s_p \:\colon \: q \in \maxexp p \longmapsto \logof{\frac qp} - \expectat p {\logof{\frac qp}} \in \sdomain p \subset \eBspace p,
 \end{equation*}
 with inverse

 \begin{equation*}
   s_p^{-1} = e_p \:\colon\:  u \in \sdomain p \mapsto \euler^{u - K_p(u)}  p \in \maxexp p \subset \pdensities
 \end{equation*}
 The atlas, $\setof{s_p \colon \sdomain p}{ p \in \pdensities}$ is affine and defines the \emph{exponential (statistical) manifold} $\pdensities$. 
 \end{definition}

We collect here various results from \cite{MR1370295,MR1704564,cena:2002,MR2396032} about additional properties of $K_{p}$.

 \begin{proposition}\label{pr:misc} Let $q = \euler^{u - K_p(u)}  p \in \maxexp p$ with $u \in \sdomain p$.
    \begin{enumerate}
\item \label{item:firsttwo} The first three derivatives of $K_p$ on $\sdomain p$ are:

 \begin{align*}
 d K_p(u)[v]&= \expectat q v \ , \\ 
 d^2 K_p(u)[v_1, v_2] &= \covat q {v_1}{v_2} \ , \\ 
 d^3 K_p(u)[v_1, v_2,v_3] &= \Cov_q(v_1,v_2,v_3) \ . \\ 
 \end{align*}

\item The random variable, $\dfrac q p -1$, belongs to $\mBspace p$ and:

 \begin{equation*}
   d {K_p(u)} [v] = \expectat p {\left( \frac q p - 1\right) v} \ .
\end{equation*}
 In other words, the gradient of $K_p$ at $u$ is identified with an element of the predual space of $\eBspace p$, \textit{viz}. $\mBspace p$, denoted by $\nabla K_p(u) = e^{u - K_p(u)} - 1=\dfrac q p -1$.
 \item \label{item:weakderiv}
The weak derivative of the map, $\sdomain p \ni u\mapsto \nabla K_p(u)
  \in \mBspace p$, at $u$ applied to $w \in B_p$ is given by:

 \begin{equation*}
 d (\nabla K_p(u)){[w]}= \frac q p \left( w- \expectat q w \right)
 \end{equation*}
and it is one-to-one at each point.
\item \label{pr:misc-1} $q/p \in L^{(\cosh-1)_*}(p)$.
   \item  \label{pr:misc-3} $B_q$ is defined by an orthogonality property:

     \begin{equation*}
     B_q = L_0^{\cosh-1}(q) = \setof{u \in L^{\cosh-1}(p)}{ \expectat p {u \frac qp}  = 0} \ .
     \end{equation*}
  \end{enumerate}
 \end{proposition}
On the basis of the above result, it appears natural to define the following \emph{parallel transports}:
\begin{definition}\ 
\begin{enumerate}
\item
The \emph{exponential transport} $\etransport p q \colon \eBspace p \to \eBspace q$ is computed as 

\begin{equation*}
\etransport p q u = u - \expectat q u \qquad u \in \eBspace p \ .
\end{equation*}
\item
The \emph{mixture transport} $\mtransport q p \colon \mBspace q \to \mBspace p$ is computed as 

\begin{equation*}
\mtransport q p v = \frac p q v \qquad v \in \mBspace q \ .
\end{equation*}
\end{enumerate}
\end{definition}
One has the following properties
\begin{proposition} Let $p,q \in \pdensities$ be given. Then
\begin{enumerate}
\item $\etransport p q$ is an isomorphism of $B_p$ onto $B_q$.
\item $\mtransport q p$ is an isomorphism of $\mBspace q$ onto $\mBspace p$.
\item The mixture transport  $\mtransport q p$ and the exponential transport $\etransport p q$ are dual of each other:  if $u \in \eBspace q$ and $v \in \mBspace p$ then

\begin{equation*}
  \scalarat q u {\mtransport p q v} = \scalarat p {\etransport q p u} v \ .
\end{equation*}
\item $\mtransport p q$ and $\etransport p q$ together transport the duality pairing: if $u \in \eBspace p$ and $v \in \mBspace p$, then

\begin{equation*}
  \scalarat p u v = \scalarat q {\etransport p q u}{\mtransport p q v} .
\end{equation*}
\end{enumerate}
\end{proposition}
We reproduce here a scheme of how the affine manifold works. The domains of the charts centered at $p$ and $q$ respectively are either disjoint or equal if $p \smile q$:

\begin{equation*}
\xymatrix{%
\maxexp p \ar[r]^{s_p}\ar@{=}[d] & \sdomain p \ar[d]_{s_q\circ s_p^{-1}} \ar[r]^{I} & \eBspace p \ar[d]^{d(s_q\circ s_p^{-1})} \ar[r]^{I} & L^{\cosh -1}(p) \ar@{=}[d] \\ 
\maxexp q \ar[r]_{s_q} & \sdomain q \ar[r]_{I} & \eBspace q \ar[r]_{I} & L^{\cosh-1}(q)
}
\end{equation*}

Our discussion of the tangent bundle of the exponential manifold is based on the concept of the velocity of a curve as in \cite[\S 3.3]{MR960687}  and it is mainly intended to underline its statistical interpretation, which is obtained by identifying curves with one-parameter statistical models. For a statistical model $p(t)$, $t \in I$, the random variable, $\dot p(t)/p(t)$ (which corresponds to the \emph{Fisher score}), has zero expectation with respect to $p(t)$, and its meaning in the exponential manifold is velocity; see \cite{MR882001} on exponential families. More precisely, let $p(\cdot) \colon I \to \maxexp p$, $I$ the open real interval containing zero. In the chart centered at $p$, the curve is $u(\cdot) \colon I \to \eBspace p$, where $p(t) = \euler^{u(t)-K_p(u(t))}  p$.

 \begin{definition}[Velocity field of a curve and tangent bundle]\  
 \begin{enumerate}
 \item\label{item:velocity}
 Assume $t \mapsto u(t) = s_p(p(t))$ is differentiable with derivative $\dot u(t)$. Define:

 \begin{equation}
   \frac{D}{dt} p(t) = \etransport p {p(t)} \dot u(t) = \dot u(t) - \expectat {p(t)} {\dot u(t)} = \derivby t {(u(t) - K_{p}(u(t))} = \derivby t {\logof{\frac{p(t)}{p}}} = \frac{\derivby t p(t)}{p(t)} \ .
 \end{equation}
 Note that $D p$ does not depend on the chart $s_p$ and that the derivative of $t \mapsto p(t)$ in the last term of the equation is computed in $L^{\Phi_*}(p)$. The curve $t \mapsto (p(t), D p(t))$ is the \emph{velocity field} of the curve.
 \item On the set $\setof{(p,v)}{p \in \pdensities, v \in \eBspace p}$, the charts:

   \begin{equation}
     s_p \colon \setof{(q,w)}{q \in \maxexp p, w \in \eBspace q} \ni (q,w) \mapsto (s_p(q),\etransport q p w) \in \sdomain p \times \eBspace p \subset \eBspace p \times \eBspace p
   \end{equation}
 define the \emph{tangent bundle}, $T\pdensities$.
 \end{enumerate}
 \end{definition}

\begin{remark}
Let $E \colon \maxexp p \to \reals$ be a $C^1$ function. Then, $E_p = E \circ e_p \colon \sdomain p \to \reals$ is differentiable and:

 \begin{equation*}
   \derivby t E(p(t)) = \derivby t E_p(u(t)) = dE_p(u(t)) \dot u(t) = dE_p(u(t)) \etransport {p(t)} p D p(t)
 \end{equation*}
\end{remark}
\subsection{Pretangent Bundle}
\label{sec:pretangent-bundle}

Let $M$ be a density and $\maxexp M$ its associated exponential manifold. Here $M$ is generic, later it will be the Maxwell distribution. All densities are assumed to be in $\maxexp M$.

\begin{definition}[Pretangent bundle $\mTof p$]
The set

\begin{equation*}
  \mTof M = \setof{(q,V)}{q \in \maxexp M, V \in \mBspace q}
\end{equation*}
together with the charts:

\begin{equation*}
  \prescript{*}{}s_p \colon \mTof M  \ni (q,V) \mapsto \left(s_p(q), \mtransport q p V\right)
\end{equation*}
is the \emph{pretangent bundle}, $\mTof M$.
\end{definition}
Let $F$ be a vector field of the pretangent bundle, $F \colon \maxexp M \to \mTof M$. In the chart centered at $p$, the vector field is expressed by

  \begin{equation*}
    F_p(u) = \mtransport {e_p(u)} p {F\circ e_p (u)} = \frac {e_p(u)} p F \circ e_p(u) \in \mBspace p, \quad u \in \sdomain M 
  \end{equation*}
If $F_p$ is of class $C^1$ with derivative $dF_p(u) \in L(\eBspace p, \mBspace p)$, for each differentiable curve $t \mapsto p(t) = \euler^{U(t) - K_p(U(t))} \cdot p$ we have $D p(t) = \derivby t \log p(t) = \dot U(t) - \expectat {p(t)} {\dot U(t)}$ and also 

\begin{equation*}
  \derivby t F_p(U(t)) = d F_p(U(t)) [\dot u(t)] = dF_p(u(t)) [\etransport {p(t)} p {D p(t)}] \in \mBspace p
\end{equation*}
\begin{definition}[Covariant derivative in $\mTof M$]
Let $F$ be a vector field of class $C^1$ of the pretangent bundle $\mTof M$, and let $G$ be a continuous vector field in the tangent bundle $\eTof M$. The \emph{covariant derivative} is the vector field $D_G F$ of $\mTof M$ defined at each $q \in \maxexp M$ by

\begin{equation*}
      (D_GF)(q) = \left. d F_q(u)\right|_{u = 0} [w], \quad w = G(q)
    \end{equation*}
\end{definition}
In the definition above the covariant derivative is computed in the mobile frame because its value at $q$ is computed using the expression in the chart centered at $q$. In a fixed frame centered at $p$ we write $s_p(q) = w$ so that $e_q(u) = e_p(u -\expectat p u + w)$, and compare the two expressions of $F$ as follows.

\begin{multline*}
  F_q(u) = \mtransport {e_q(u)} q F \circ e_q(u) = \\ \mtransport {e_q(u)} q \mtransport p {e_p(u - \expectat p u + w)} \mtransport {e_p(u - \expectat p u + w)} p F \circ e_p(u - \expectat p u + w) = \\ \mtransport p q  \mtransport {e_p(u - \expectat p u + w)} p F \circ e_p(u - \expectat p u + w) = \mtransport p q  F_p (u - \expectat p u + w). 
\end{multline*}
Derivation in the direction $v \in \eBspace q$ gives

\begin{equation*}
  d F_q(u) [v] = \mtransport p q  d F_p (u - \expectat p u + w) [v - \expectat p v],
\end{equation*}
hence $d F_q (u) = \mtransport p q dF_p(\etransport q p u + w) \etransport q p$ and at $u=0$, we have $dF_q(0) = \mtransport p q dF_p(w) \etransport q p$. It follows that the covariant derivative in the fixed frame at $p$ is

\begin{equation*}
  (D_G F)(q) = \mtransport p q dF_p(s_p(q)) \etransport q p G(q). 
\end{equation*}
The tangent and pretangent bundle can be coupled to produce the vector bundle of order 2 defined by

\begin{equation*}
  (\prescript{*}{}T \times T)\maxexp M = \setof{(p,v,w)}{p \in
    \maxexp M, v \in \eBspace p, w \in \mBspace p}
\end{equation*}
with charts 

\begin{equation*}
  (q,v,w) \mapsto (s_p(q), \etransport q p v, \mtransport q p w)
\end{equation*}
and the duality coupling:

\begin{equation*}
  (q,v,w) \mapsto \scalarat q {\etransport q p v}{\mtransport q p w}
\end{equation*}
\begin{proposition}[Covariant derivative of the duality coupling] \label{prop:der.duality}
Let $F$ be a vector field of $\mTof M$, and let $G, X$ be vector fields of $\eTof M$, $F, G$ of class $C^1$ and $X$ continuous.

\begin{equation*}
  D_X \scalarof F G = \scalarof {D_X F} G + \scalarof F {D_X G}
\end{equation*}
\end{proposition}
\begin{proof}
Consider the real function $\mathcal E \ni q \mapsto \scalarof F G (q)
= \expectat q {F(q)G(q)}$ in the chart centered at any $p \in \maxexp M$:

\begin{multline*}
  \sdomain p  \ni u \mapsto \expectat {e_p(u)} {F(e_p(u))G(e_p(u))} = \\ \expectat p
  {\mtransport {e_p(u)} p {F \circ e_p(u)} \etransport {e_p(u)} p {G \circ e_p(u)}}
  = \expectat p {F_p(u)G_p(u)}, 
\end{multline*}
and compute its derivative at 0 in the direction $X(p)$.
\end{proof}

We refer to \cite{MR3126029,MR3130268} for further details on the geometric structure, namely the Hilbert bundle, the tangent mapping of an homeomorphism, the Riemannian Hessian. We now turn to a basic example.

\subsection{Kullback-Leibler Divergence}
\label{sec:kullb-leibl-diverg}

The Kullback-Leibler divergence \cite{MR0039968} on the exponential manifold $\Maxexp$ is the mapping

\begin{equation*}
  \mathbf{D}\: \colon\:  (q_1,q_2) \in \mathcal E \times \mathcal E \longmapsto \KL {q_1}{q_2} = \expectat {q_1} {\logof{\frac {q_1}{q_2}}}. 
\end{equation*}
Notice that, if $q_{2} = \euler^{u - K_{q_{1}}(u)} \cdot q_{1} \in \maxexp {q_{1}}$, $u \in \sdomain {q_1}$, then $K_{q_{1}}(u) = \expectat {q_{1}} {\logof{q_{1}/q_{2}}}$ is the expression in the chart centered at $q_{1}$ of the marginal Kullback-Leibler divergence $q_2 \mapsto \KL {q_1}{q_{2}}$. Therefore, the Kullback-Leibler divergence is non-negative valued and zero if and only if $q_2=q_1$ because of Theorem \ref{prop:maxexp-pormanteau}, item (\ref{prop:maxexp-pormanteau-4}). Its expression in the chart centered at a generic $p \in \Maxexp$ is

\begin{equation*}
  D_p \:\colon\: (u_1,u_2) \in \sdomain {p} \times \sdomain {p}  \longmapsto K_p(u_2) - K_p(u_1) - dK_p(u_1)[u_2 - u_1] \ ,
\end{equation*}
which is the Bregman divergence \cite{MR0215617} of the convex function $K_p \colon \sdomain p \to \reals$. 

It follows from Proposition \ref{prop:cumulant}.(\ref{prop:cumulant4}) that it is $C^\infty$ jointly in both variables and, moreover, analytic with

\begin{equation*}
  D_p(u_1,u_2) = \sum_{n \ge 2} \frac 1{n!}d^n K_p(u_1) [(u_1 - u_2)^{\otimes n}], \quad \normat {\Phi,p} {u_1 - u_2} < 1.
\end{equation*}
This regularity result is to be compared with what is available when the restriction, $q_1 \smile q_2$, is removed, {\em i.e.}, the semi-continuity \cite[\S 9.4]{MR2401600}.

The partial derivative of $D_p$ in the first variable, that is the derivative of $u_1 \mapsto D_p(u_1,u_2)$, in the direction $v \in \eBspace p$ is

\begin{equation*} d (u_1 \mapsto  D_p(u_1,u_2)) [v]= - d^2 K_p(u_1)[u_2-u_1,v] = - \covat {q_1} {\log\frac{q_2}{q_1}}{v},
\end{equation*}
with $ \quad q_i = \operatorname{e}_p(u_i)$, $i = 1,2$. If $v = \etransport {q_1} p w$, we have $- \covat {q_1} {\log\frac{q_2}{q_1}}{v} = \covat {q_1} {\log\frac{q_1}{q_2}}{w}$, so that we can compute both the covariant derivative of the partial functional $q \mapsto \KL {q} {q_2}$ and its gradient as

\begin{align*}
  D_{w}(q \mapsto \KL {q} {q_2}) &= \covat {q} {\log\frac{q}{q_2}}{w} = \expectat {q} {\left(\log\frac{q}{q_2} - \KL q {q_2}\right)w} , \\ \nabla (q \mapsto \KL {q} {q_2}) &= \log\frac{q}{q_2} - \KL q {q_2}.
\end{align*}
The negative gradient flow is 

\begin{equation*}
  \label{eq:gradflowd1}
  \derivby t \log\frac{q(t)}{q_2} = - \log\frac{q(t)}{q_2} + \KL {q(t)} {q_2}.
\end{equation*}
As $\derivby t \left(\euler^t \log\frac{q(t)}{q_2}\right) = \euler^t \KL {q(t)} {q_2}$, for each $t$ the random variable

\begin{equation*}
  \euler^t \log\frac{q(t)}{q_2} - \log\frac{q(0)}{q_2} = \euler^t \logof{\frac{q(t)}{q_2}\left(\frac{q_2}{q(0)}\right)^{\euler^{-t}}} = \euler^t \logof{q(t)\frac{q_2^{\euler^{-t}-1}}{q(0)^{\euler^{-t}}}} 
\end{equation*}
is constant, so that $q(t) \propto q(0)^{\euler^{-t}}q_2^{1-\euler^{-t}}$. It is the exponential arc of $q(0) \smile q_2$ in an exponential time scale.

The partial derivative of $D_p$ in the second variable, that is the derivative of $u_2 \mapsto D_p(u_1,u_2)$, in the direction $v \in \eBspace p$ is

\begin{equation*}
  d (u_2 \mapsto  D_p(u_1,u_2)) v = d K_p(u_2) [v] -d K_p(u_1) [v] = \expectat {q_2} {v} - \expectat {q_1} {v} \ ,
\end{equation*}
with $ \quad q_i = \operatorname{e}_p(u_i)$, $i = 1,2$. If $v = \etransport {q_2} p w$, we have $\expectat {q_2} {v} - \expectat {q_1} {v} = \expectat {q_2} {w} - \expectat {q_1} {w}$, so that we can compute both the covariant derivative of the partial functional $q \mapsto \KL {q_1} {q}$ and its gradient as

\begin{align*}
  D_{w}(q \mapsto \KL {q_1} {q}) &= \expectat {q} {w} - \expectat {q_1} {w} = \expectat q {\left(1 - \frac {q_1}q\right) w}, \\ \nabla (q \mapsto \KL {q_1} {q}) &= 1 - \frac {q_1} {q} \ .
\end{align*}
The negative gradient flow is 

\begin{equation*}
  \label{eq:gradflowd2}
  \frac{\dot q(t)}{q(t)} = -\left(1 - \frac {q_1}{q(t)}\right) = \frac{q_1-q(t)}{q(t)} \ ,
\end{equation*}
whose solution starting at $q_0$ is $q(t) = q_1 + (q_0-q_1) \euler^{-t}$. It is a mixture model in an exponential time scale.

\section{Gaussian space}
\label{chap:gaussian-space}

In this Section the sample space is $\reals^n$, $M$ denotes the standard $n$-dimensional Gaussian 
  density (we denoted it $M$ because of James Clerk Maxwell (1831 -- 1879))

\begin{equation}\label{Maxw}
M(x)=\dfrac{1}{(2\pi)^{n/2}}\exp\left(-\dfrac{|x|^{2}}{2}\right), \qquad x \in \reals^{n}\end{equation}
 and $\maxexpM$ is the exponential manifold containing $M$. We recall that the Orlicz space $\Lexp M$ is defined with the Young function $\Phi:=x \mapsto \cosh x - 1$. The following propositions depend on the specific properties of the Gaussian density $M$. They do not hold in general.

\begin{proposition} \ 
1. The Orlicz space $\Lexp M$ contains all polynomials with degree up to 2.  

2. The Orlicz space $\LlogL M$ contains all polynomials.
\end{proposition}
\begin{proof} 1. If $f$ is a polynomial of degree $d \le 2$ then

  \begin{equation*}
 \int_{\reals^{n}} \euler^{\alpha f(x)} M(x) \ dx =\dfrac{1}{(2\pi)^{n/2}} \int_{\reals^{n}}  \euler^{\alpha f(x) - \frac12 \absoluteval x^2} \ dx   
  \end{equation*}
and the latter is finite for all $\alpha$ such that $\alpha \hessianof f - I$ is negative definite.

2. The result comes from the fact that all polynomials belong to $L^{2}(M)$ and one has $L^2(M) \subset \LlogL M$.\end{proof}

\subsection{Boltzmann-Gibbs Entropy}
\label{sec:boltzmann-entropy}

While the Kullback-Leibler divergence $\KL{q_1}{q_2}$ of Sec. \ref{sec:kullb-leibl-diverg} is defined and finite if the densities $q_1$ and $q_2$ belong to the same exponential manifold, the \emph{Boltzmann-Gibbs entropy}  (BG-entropy in the sequel)

$$H(q) = - \expectat q {\logof q}$$ could be either non defined or infinite, precisely $-\infty$, everywhere on some exponential manifolds, or finite everywhere on other exponential manifolds. 

\begin{proposition}
Assume $p \smile q$, Then:
\begin{enumerate}
\item For each $a \ge 1$, $\log p \in L^a(p)$ if, and only if, $\log q \in L^a(q)$.
\item $\log p \in \Lexp p$ if, and only if, $\log q \in \Lexp q$.
\end{enumerate}
\end{proposition}
\begin{proof}
If $p, q$ belong to the same exponential manifold, we can write $q = \euler^{u - K_p(u)}\cdot p$  and, from Theorem \ref{prop:maxexp-pormanteau}.(\ref{prop:maxexp-pormanteau-4}), we obtain $\log q - \log p = u - K_p(u) \in \Lexp p = \Lexp q$, so that $\log q \in L^{a}(q)$ if, and only if, $\log p \in L^{a}(p)$, $a \ge 1$, and $\log q \in \Lexp q$ if, and only if, $\log p \in \Lexp p$.
\end{proof}

In order to obtain a smooth function, we study the BG-entropy $H(q)$ on all manifolds $\mathcal E$, such that for at least one, and, hence for all, $p \in \mathcal E$, it holds $\logof p \in \Lexp p$. In such an exponential manifold we can write

\begin{equation*} 
  \KL q p = - H(q) - \expectat q {\log p} \ ,
\end{equation*}
so that $H(q) \le - \expectat q {\log p}$.

For example, it is the case when the reference measure is finite and $p$ is constant. Another notable example is the Gaussian case, where the sample space is $\reals^n$ endowed with the Lebesgue measure and $p(x) \propto \exp{-\frac{|x|^2}{2}}$. In such  case $\int \cosh (\alpha |x|^2) \expof{-|x|^2/2} \ dx < +\infty$ if $0 < \alpha < 1/2$.

We investigate here the main properties of the BG-entropy in this context. First, one has
\begin{proposition}
The BG-entropy is a smooth real function on  the exponential manifold $\mathcal E$. Namely, if $p \in \mathcal E$ then 

$$H_{p}\::\:u \in \sdomain {p} \mapsto H\circ e_p(u)$$ 
is a $C^\infty$ real function. Moreover, its derivative in the direction $v$ equals

\begin{equation*}
  d H_p(u) [v]= - \covat q {u+\log p}v \ ,
\end{equation*}
where $q = e_p(u)$.
\end{proposition}

\begin{proof}
As

\begin{equation*}
  - \log q = - u + K_p(u) - \log p = - (u + \log p + H(p)) + K_p(u) + H(p) \in \Lexp p,
\end{equation*}
with $- (u + \log p + H(p)) \in B_p$, the representation of the BG-entropy in the chart centered at $p$ is

\begin{align*}
  H_p(u) &= H\circ e_p(u) \notag \\ &= - \expectat {e_p(u)}{u + \log p + H(p)}  + K_p(u) + H(p) \notag \\  &= - dK_p(u) \left[u + \log p + H(p))\right] + K_p(u) + H(p) \notag \\ &= (K_p(u)-d_uK_p(u)) - dK_p(u)[\log p + H(p)] + H(p),
\end{align*}
hence, $u \mapsto H_p(u)$ is a $C^\infty$ real function. Notice that $K_p(u)-d_uK_p(u) \le 0$, hence $H(q) \le \expectat q {\log p}$, as we already know.
The derivative of $H_p$ in the direction $v$ equals

\begin{equation*}
  d H_p(u) [v]= - d^2 K_p(u) [(u + \log p + H(p)),v] + d_v K_p(u) = - \covat q {u+\log p}v \ .
\end{equation*}
\end{proof}
Notice that, for $u = 0$ and $q=e_{p}(u),$ we have $\expectat q v = \expectat p v = 0$, hence

\begin{equation*}
  dH_p(0)[v] = - \covat p {\log p}v = - \expectat p {\logof p  v}\ . 
\end{equation*}
\begin{proposition}
The gradient field $\nabla H$ over $\mathcal E$ can be identified, at each $p$, with random variable $ \nabla H (p) \in \eBspace p \subset \mBspace q$,

\begin{equation*}
  \nabla H(p) = -(\logof p + H(p)) .
\end{equation*}
\end{proposition}

\begin{proof}
The covariant derivative $D_GH$ at $p \in \mathcal E$ with respect to the vector field $G $ defined on $\mathcal{E}$ with $G(p) \in B_p$ and $p \in \mathcal E$ is

\begin{equation*}
D_G H(p) = - \expectat p {\logof p  G(p)} = - \scalarat p {\logof p + H(p)}{G(p)} \ .
\end{equation*}

The gradient field $\nabla H$ over $\mathcal E$, is then defined by $D_GH(p) = \scalarat p {\nabla H(p)}{G(p)}$. This justifies the identification with the random variable $ \nabla H (p)=-(\logof p + H(p)) \in \eBspace p \subset \mBspace q$.\end{proof}

\begin{remark} The equation $\nabla H(p) = 0$ implies $ \log p = - H(p)$, hence $p$ has to be constant and this requires it is the finite reference measure $\mu$.\end{remark}

We refer to \cite{MR3130268} for more details on the BG-entropy and in particular on the evolution of $H$ on $C^{1}$ curve in $\Maxexp$
of the type $I \ni t \mapsto f_t$. 
 
\section{Boltzmann equation}\label{sec:boltzmann-1} 

We consider a space-homogeneous Boltzmann operator as it is defined, for example, in \cite{MR1942465} and \cite{MR2409050}. We retell the basic story in order to introduce our notations and the IG background. Orlicz spaces as a setting for Boltzmann's equation have been recently proposed by \cite{MR3205750}, while the use of exponential statistical manifolds has been suggested in \cite[Example 11]{MR3126029} and sketched in \cite[Sec 4.4]{MR3130268}. We start with an improvement of the latter, a few repetitions being justified by consistency between this presentation and \cite[1.3, 4.5--6]{MR2409050}, compare also Prop. \ref{prop:conditioning} below.

\subsection{Collision kinematics}

We review our notations, see our Fig. \ref{fig:1}, cf. \cite[Fig. 1]{MR2409050}.
\begin{figure}
\centering
\includegraphics[scale=.45,viewport = 68 73 462 442]{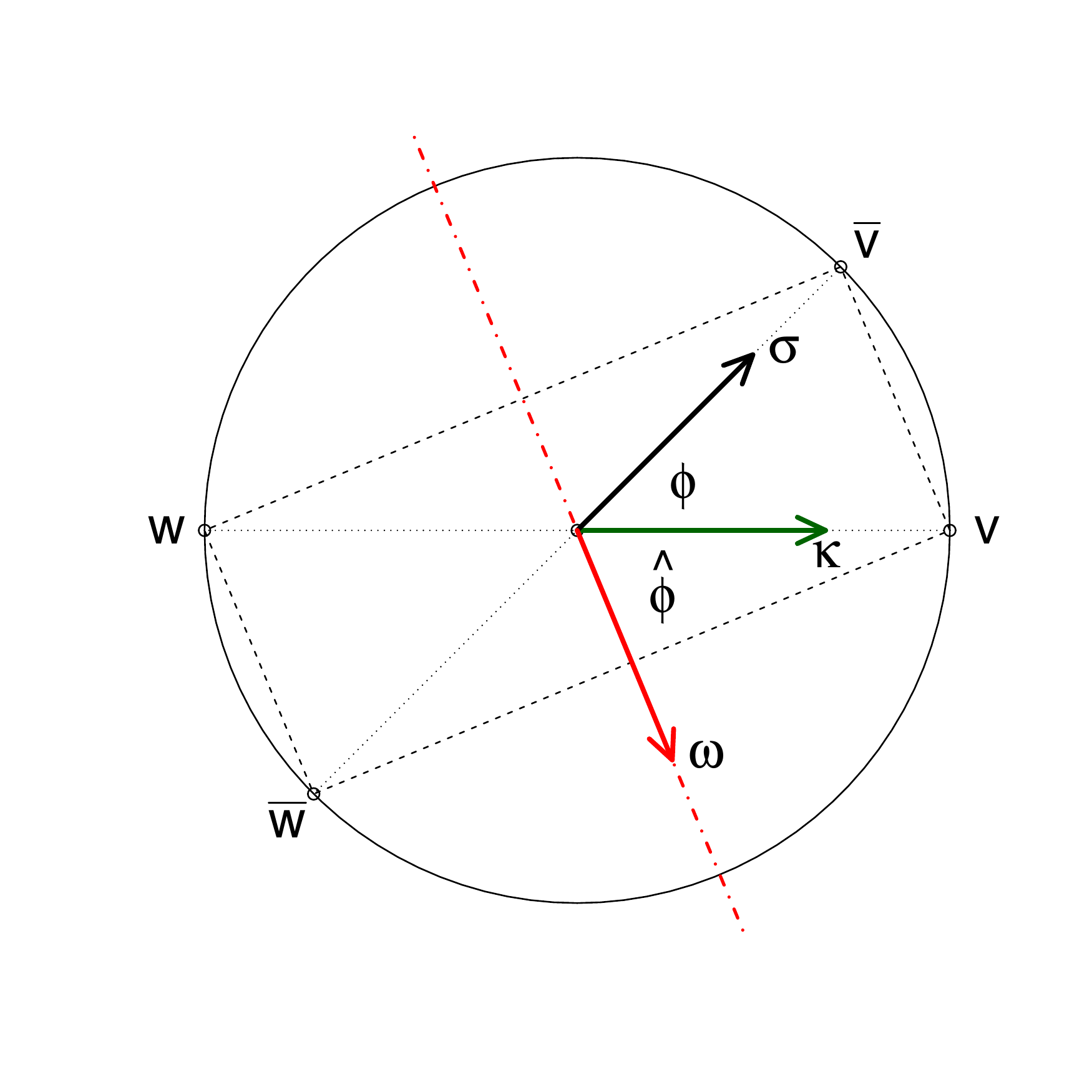} \qquad
\includegraphics[scale=.45,viewport = 59 86 474 432]{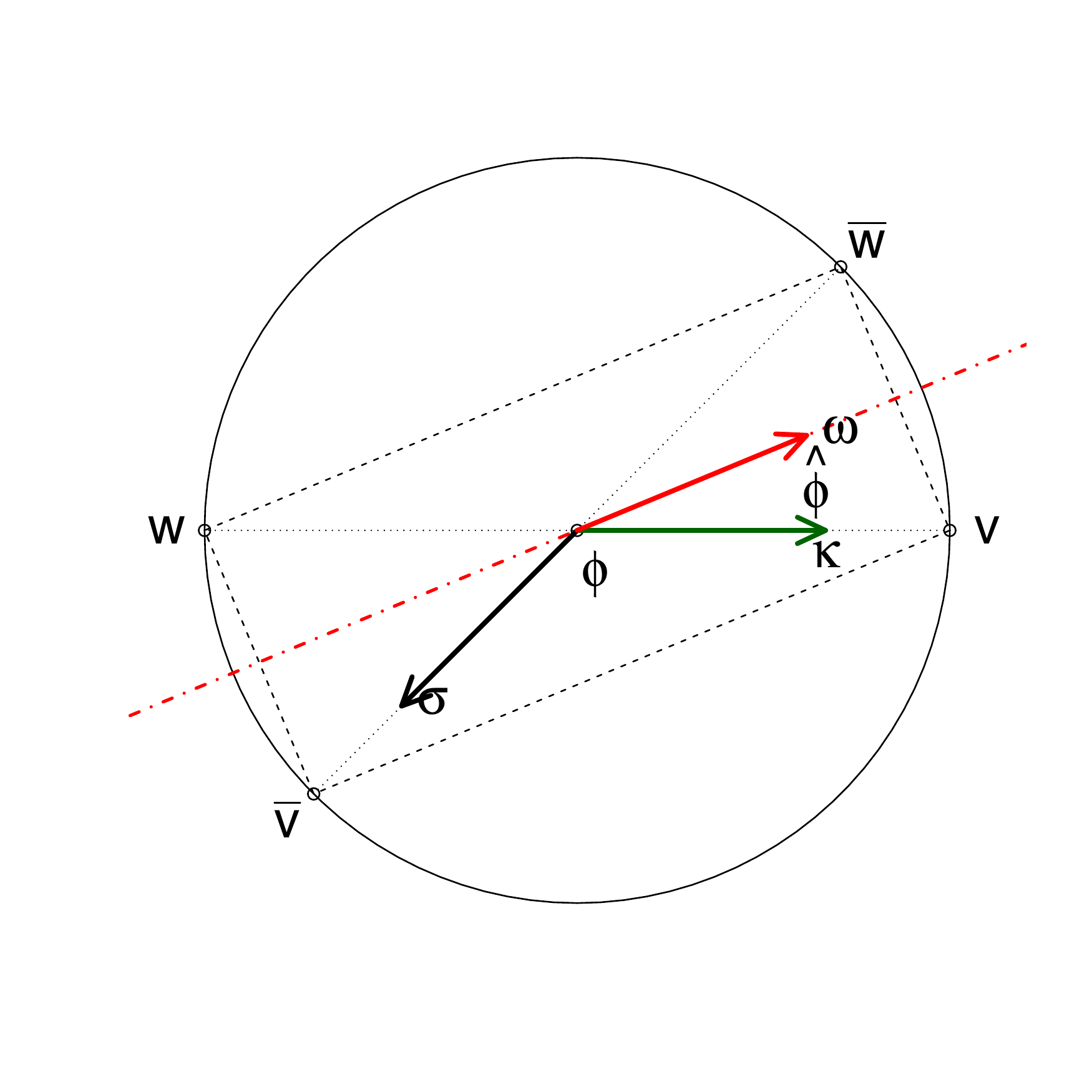}
\caption{Elastic collision for two opposite $\sigma$'s corresponding to the two labeling of the velocities after collision. Velocities $v$ and $w$ are before collision, $\widebar v$ and $\widebar w$, after collision. The unit vectors are $\kappa  = \versof{v-w}$, $\sigma = \versof{\bar v - \widebar w}$, $\omega = \versof{v - \widebar v} = \versof{\kappa - \sigma}$. The red dotted line represents the space generated by $\pm \omega$.The angles are  given by $\cos \phi = \kappa \cdot \omega$, $\phi \in [0,\pi]$. As $\omega \cdot (\sigma + \kappa) = 0$, the angles are related by $\hat\phi = (\pi - \phi)/2$.}
\label{fig:1}  
\end{figure}
We denote by $v, w \in \reals^3$ the velocities before collision, while the velocities after collision are denoted by $\widebar v, \widebar w \in \reals^3$. The quadruple $(v,w,\widebar v, \widebar w) \in (\reals^{3})^4$, is assumed to satisfy the \emph{conservation laws} 

\begin{align}
 F_1(v,w,\widebar v, \widebar w) &= v + w - (\widebar v + \widebar w) = 0, \label{eq:momentum}\\ 
 F_2(v,w,\widebar v, \widebar w) &= \absoluteval {v}^2 +   \absoluteval {w}^2 - (\absoluteval {\widebar v}^2 + \absoluteval {\widebar w}^2) = 0 ,  \label{eq:energy} 
\end{align}
which define an algebraic variety $\mathcal M$ that we expect to have dimension $12-(3+1)=8$. The Jacobian matrix of the four defining Eq.s \eqref{eq:momentum} and \eqref{eq:energy} is

\begin{equation}\label{eq:jacobianM}
\bordermatrix[{[}{]}]{%
  & v_1 & v_2 & v_3 & w_1 & w_2 & w_3 & \widebar v_1 & \widebar v_2 & \widebar v_3 & \widebar w_1 & \widebar w_2 & \widebar w_3 \cr
\eqref{eq:momentum}_1 & 1   & 0   & 0  & 1   & 0   & 0   & -1   & 0   & 0  & -1   & 0   & 0 \cr 
\eqref{eq:momentum}_2 & 0   & 1   & 0  & 0   & 1   & 0   & 0   & -1   & 0  & 0   & -1   & 0 \cr
\eqref{eq:momentum}_3 & 0   & 0   & 1  & 0   & 0   & 1   & 0   & 0   & -1  & 0   & 0   & -1  \cr
\eqref{eq:energy} & 2v_1 & 2v_2 & 2v_3 & 2w_1 & 2w_2 & 2w_3 & -2\widebar v_1 & -2\widebar v_2 & -2\widebar v_3 & -2\widebar w_1 & -2\widebar w_2 & -2\widebar w_3 \cr}  
\end{equation}
The Jacobian matrix in Eq. \eqref{eq:jacobianM} has in general position full rank, and rank 3 if $v=w=\bar v=\bar w$. We denote by $\mathcal M_*$ the 8 dimensional manifold $\mathcal M \setminus \set{v=w=\widebar v = \widebar w}$. In the sequel, for $v \neq w$, we set 

$$\versof{v-w}=\dfrac{v-w}{|v-w|}.$$

From \eqref{eq:momentum} and \eqref{eq:energy} it follows the conservation of both the scalar product, $v \cdot w = \bar v \cdot \bar w$ and of the norm of the difference, $\absoluteval{v-w} = \absoluteval{\widebar v - \widebar w}$, so that all the vectors of the quadruple lie on a circle with center $z = (v+w)/2 = (\bar v + \bar w) / 2$ and are the four vertexes of a rectangle. If $v = w$ then $\widebar v = \widebar w$, and also $v=w=\widebar v = \widebar w$ as the circle collapse to one point, hence we have $\mathcal M_* = \mathcal M \setminus \set{v=w} = \mathcal M \setminus \set{\widebar v= \widebar w}$. 

There are various explicit and interesting parametrizations of $\mathcal M_*$ available. 

An elementary parametrization consists of any algebraic solution of Eq.s \eqref{eq:momentum} and \eqref{eq:energy} with respect to any of the free 8 coordinates. Other parametrizations are used in the literature, see classical references on the Boltzmann equation, e.g. \cite{MR1942465}. 

A \emph{first} parametrization is

\begin{equation} \label{eq:sigma-par}
   (u,v,\sigma) \in (\reals^3 \times \reals^3)_* \times \Sd\mapsto \left(u,v,\hat A_{\sigma}(u,v)\right) \in \mathcal M_* \subset (\reals^3)^4,
\end{equation}
where $\Sd = \setof{\sigma \in \reals^3}{\,\absoluteval{\sigma}=1}$ and the \emph{collision transformation} $\hat A_\sigma \colon (v, 
w) \mapsto ( v,  w) = (v_\sigma,w_\sigma)$ is:

\begin{equation} \label{eq:bar-A-y-matrix}
\hat A_\sigma \colon \left\{\begin{aligned}
  v_\sigma &= \frac{v+w}2 + \frac{\absoluteval{v-w}}2 \sigma \\
  w_\sigma &= \frac{v+w}2 - \frac{\absoluteval{v-w}}2 \sigma
\end{aligned}\right..
\end{equation}
Viceversa, on $\mathcal M_*$ the collision transformation depends on the unit vector $\sigma = \versof{  v -   w} \in \Sd$, while the other terms depend on the collision invariants, as $\absoluteval{v-w}^2 = 2(\absoluteval{v}^2+\absoluteval{w}^2)-\absoluteval{v+w}^2$. In conclusion, the transformation in Eq. \eqref{eq:sigma-par} is 1-to-1 from $(\reals^3\times \reals^3)_* \times \Sd$ to $\mathcal M_*$, where $(\reals^3\times \reals^3)_* = \setof{(u,v)\in \reals^3\times\reals^3}{v\ne w}$.

A \emph{second} parametrization of $\mathcal M_*$ is obtained  using the common space of two parallel sides of the velocity's rectangle, $\spanof{v - \widebar v} = \spanof{\widebar w - w}$, so that $v - \widebar v = \widebar w - w = \Pi(v - w)$, where $\Pi$ is the orthogonal projection on the subspace. Viceversa, given any $\Pi$ in the set $\Pi(1)$ of projections of rank 1, the mapping

\begin{equation}\label{eq:A-x-matrix}
  A_\Pi =
  \begin{bmatrix}
    (I - \Pi) & \Pi \\ \Pi & (I - \Pi)
  \end{bmatrix},
\quad
\left\{\begin{aligned}
    v_\Pi &= v - \Pi(v - w) = (I - \Pi)v + \Pi w  \\
    w_\Pi &= w + \Pi (v - w) = \Pi v + (I - \Pi)w
  \end{aligned}\right..
\end{equation}
The components in the direction of the image of $\Pi$ are exchanged, $\Pi v_\Pi = \Pi w$ and $\Pi w_\Pi = \Pi v$, while the orthogonal components are conserved. If $\omega$ is any of the two unit vectors such that $\Pi = \omega \otimes\omega'$, the matrix $A_{\Pi}=A_{\omega\otimes\omega'}$ does not depend on the direction of $\omega \in \Sd$. Notice that $A_\Pi' = A_\Pi$ and $A_\Pi A_\Pi' = I_6$, that is $A_\Pi$ is an orthogonal symmetric matrix. This parametrization uses the set $\Pi(1)$ of projection matrices of rank 1,

\begin{equation} \label{eq:omega-par}
  (\reals^3 \times \reals^3)_* \times \Pi(1) \ni (u,v,\Pi) \mapsto \left(u,v, A_{\Pi}(u,v)\right) \in \mathcal M_* \in (\reals^{3})^4 \ .
\end{equation}

The $\sigma$-parametrization \eqref{eq:sigma-par} and the $\Pi$-parametrization \eqref{eq:omega-par} are related as follows. Given the unit vector $\kappa = \versof{v-w}$, the parameters $\Pi$ and $\sigma$ in Eq.s \eqref{eq:A-x-matrix} and \eqref{eq:bar-A-y-matrix} are in 1-to-1 relation as

\begin{equation}\label{eq:sigma.vs.Pi}
\sigma = (I - 2\Pi)\kappa,  \quad \Pi = \versof{\kappa - \sigma}\otimes\versof{\kappa - \sigma}.  
\end{equation}

The transition map from the parametrization \eqref{eq:sigma-par} to the parametrization \eqref{eq:omega-par} is

\begin{equation*}
(v,w,\sigma) \mapsto (v,w,(\versof{\versof{v-w}-\sigma})\otimes(\versof{\versof{v-w}-\sigma}),  
\end{equation*}
while the inverse transition is

\begin{equation*}
  (v,w,\Pi) \mapsto (v,w,(I-2\Pi)\versof{v-w}).
\end{equation*}

\subsection{Uniform probabilities on $\Sd$ and on $\Pi(1)$}

Let $\mu$ be the uniform probability on $\Sd$, computed, for example, in polar coordinates by

\begin{equation} \label{eq:mu-integral}
  \integrald {\Sd} {f(\sigma)} {\mu(d\sigma)} = \frac1{4\pi} \int_0^{\pi} \sin\varphi\ d\varphi \int_0^{2\pi}  \ f(\sin\varphi \cos\theta\ u_1 + \sin\varphi\sin\theta\ u_2 + \cos\varphi\ u_3)d\theta,
\end{equation}
where $u_1,u_2,u_3$ is any orthonormal basis of $\reals^3$ that is $U = [u_1 u_2 u_3] \in \operatorname{SO}(3)$. In such a way, $U \colon \Sd \mapsto \Sd$  and the right hand side of Eq. \eqref{eq:mu-integral} does not depend on $U$.

As the mapping $\omega \in \Sd \mapsto \Pi = \omega\otimes\omega  \in \Pi(1)$ is a 2-covering, we define the image $\nu$ of $\mu$ by the equation

\begin{equation}\label{eq:nu-integral}
  \int_{\Pi(1)} g(\Pi)\ \nu(d\Pi) = \int_{\Sd} g(\omega\otimes\omega)\ \mu(d\omega) = 2 \int_{\setof{\sigma \in \Sd}{\kappa\cdot\sigma > 0}} g(\omega\otimes\omega)\ \mu(d\omega),  
\end{equation}
where $\kappa \in \Sd$ is any unit vector used to split $\Sd$ in two parts, $\setof{\sigma}{\kappa\cdot\sigma > 0}$ and $\setof{\sigma}{\kappa\cdot\sigma < 0}$.
Eq. \eqref{eq:nu-integral} defines a probability $\nu$ on $\Pi(1)$ such that we have the invariance

\begin{equation*}
  \int _{\Pi(1)} g(U{\Pi}U')\ \nu(d\Pi) = \int _{\Pi(1)} g(\Pi)\ d\nu(\Pi), \quad  U \in \operatorname{SO}(3).
\end{equation*}

Let us compute the image $T^*_\kappa \mu$ of the uniform measure $\mu$ under the action of the transformation $T_\kappa \colon \sigma \mapsto \omega = \versof{\kappa-\sigma}$. If in Eq. \eqref{eq:mu-integral} we take $u_3 = \kappa$, that is an orthonormal basis $(u_1, u_2, \kappa) $, then, for $\phi,\theta$ such that $\sigma = \sin\phi \cos\theta\ u_1 + \sin\phi\sin\theta\ u_2 + \cos\phi\ \kappa$ and $\hat\phi = (\pi - \phi)/2$ (see Fig. \ref{fig:1}) we have

\begin{align}
  \omega = \versof{\kappa - \sigma} &= \sin\hat\phi \cos\theta\ u_1 + \sin\hat\phi\sin\theta\ u_2 + \cos\hat\phi\ \kappa \notag \\ 
&= \sinof{\frac{\pi-\phi}2} \cos\theta\ u_1 + \sinof{\frac{\pi-\phi}2} \sin\theta\ u_2 + \cosof{\frac{\pi-\phi}2}\ \kappa \notag \\ 
&= \cosof{\frac{\phi}2} \cos\theta\ u_1 + \cosof{\frac{\phi}2} \sin\theta\ u_2 + \sinof{\frac{\phi}2}\ \kappa 
\end{align}
and for all integrable $f \colon \reals^3 \to \reals$ one has

\begin{align}
\integrald {\Sd} {f(\omega)} {T^*_\kappa(d\omega)} &=
  \integrald {\Sd} {f\left(\versof{\kappa - \sigma}\right)} {\mu(d\sigma)} \notag \\ &= 
\frac1{4\pi} \int_0^{\pi} \sin\phi\ d\phi \int_0^{2\pi} d\theta\ f\left(\cosof{\frac{\phi}2} \cos\theta\ u_1 + \cosof{\frac{\phi}2} \sin\theta\ u_2 + \sinof{\frac{\phi}2}\ \kappa\right) \notag \\
&= \frac1{4\pi} \int_0^{\pi/2} 2 \sin(2\phi)\ d\phi \int_0^{2\pi} d\theta\ f\left(\cos\phi \cos\theta\ u_1 + \cos\phi \sin\theta\ u_2 + \sin\phi\ \kappa\right) \notag \\
&= \frac1{4\pi} \int_0^{\pi/2} 4 \cos \phi \sin(\phi)\ d\phi \int_0^{2\pi} d\theta\ f\left(\cos\phi \cos\theta\ u_1 + \cos\phi \sin\theta\ u_2 + \sin\phi\ \kappa\right) \notag \\ 
&= \integrald {\setof{\sigma\in \Sd}{\kappa\cdot\omega \ge 0}} {f(\omega)} {4(\kappa\cdot\omega)\ \mu(d\omega)}, \label{eq:sigma-to-omega}
\end{align}
compare \cite[4.5]{MR2409050}. 

In particular, for a symmetric function, $f(\omega)=f(-\omega)$, we have

\begin{equation} \label{eq:sigma-to-omega-sym}
  \integrald {\Sd} {f(\omega)} {T^*_\kappa(d\omega)} = \integrald {\Sd} {f(\versof{\kappa-\sigma})} {\mu(d\sigma)} = \integrald {\Sd} {f(\omega)} {2\absoluteval{\kappa\cdot\omega}\mu(d\omega)}. 
\end{equation}

It follows, for each integrable $g \colon \Pi(1) \to \reals$, that

\begin{equation} \label{eq:integral-sigma-to-Pi}
  \int_{\Sd} g((\versof{\kappa-\sigma})\otimes(\versof{\kappa-\sigma}))\ \mu(d\sigma) = \int_{\Sd} g(\omega\otimes\omega)\ 2 \absoluteval{\kappa\cdot\omega} \ \mu(d\omega)
\end{equation}
Notice that if we integrate Eq. \eqref{eq:sigma-to-omega} with respect to $\kappa$ we obtain

\begin{equation} \label{eq:kappa.minus.sigma}
  \iint_{\Sd\times \Sd} f(\versof{\kappa-\sigma})\ \mu(d\sigma)\mu(d\kappa) = \integrald {\Sd} {f(\sigma)} {\mu(d\sigma)},
\end{equation}
because

\begin{equation}
  \integrald {\setof{\kappa\in \Sd}{ \kappa\cdot \sigma \geq 0}} {\kappa \cdot \sigma  } {\mu(d\kappa)} 
  =\frac14.
\end{equation}

\subsection{Conditioning on the collision invariants}

Given a function $g \colon \reals^3 \times \reals^3$, Eq. \eqref{eq:bar-A-y-matrix} shows that the function 

\begin{equation}
  \hat g \colon (u,v) \mapsto \integrald {\Sd} {g(v_\sigma,w_\sigma)} {\mu(d\sigma)} = \integrald {\Sd} {g\left(\hat A_\sigma(v,w)\right)}{\mu(d\sigma)} 
\end{equation}
depends on the collision invariants only. This, in turn, implies that $\hat g$ is the conditional expectation of $g$ with respect of the collision invariants under any probability distribution on $\reals^3\times\reals^3$ such that the collision invariants and the unit vector of $\sigma$ are independent, the unit vector $\sigma$ being uniformly distributed. See below a more precise statement in the case of the Gaussian distribution. 

 On the sample space $(\reals^3,dv)$, let $M$ be the standard normal density defined in \eqref{Maxw} (the Maxwell density). As, for all $\Pi \in \Pi(1)$, $A_\Pi A'_\Pi = A_\Pi A_\Pi = I_6$, in particular $\absoluteval{\det A_\Pi} = 1$, we have for each $(V,W) \sim M \otimes M$ (i.e. $V,W$ are i.i.d. with distribution $M$) that

\begin{equation}
  A_\Pi (V,W) = (V_\omega,W_\omega) \sim (V,W).
\end{equation}
Under the same distributions, the random variables $\frac{V+W}{\sqrt{2}}$, $\frac{\absoluteval{V-W}}{\sqrt{2}}$, $\versof{V-W}$, are independent, with distributions given by

\begin{equation}
\frac{V+W}{\sqrt{2}} \sim \text{N}\left(0_3,I_3\right), \quad \frac{\absoluteval{V-W}^2}2 \sim \chi^2(3), \quad \versof{V-W} \sim \mu,
\end{equation}
respectively. Hence, given any $S \sim \mu$ such that$\frac{V+W}{\sqrt{2}}$, $\frac{\absoluteval{V-W}}{\sqrt{2}}$, $S$, are independent, we get

 \begin{equation}
   \hat A_S(V,W) \sim (V,W) \ .
 \end{equation}
This equality of distribution generalizes the equality of random variables $\hat A_{\versof{V-W}}(V,W) = (V,W)$. We state the results obtained above as follows.

The image distribution of $M \otimes M \otimes U$ induced on $(\reals^3)^4$ by the parametrization in Eq. \eqref{eq:sigma-par} is supported by the manifold $\mathcal M_*$. Such a distribution has the property that the projections on both the first two and the last two components are $M \otimes M$. The joint distribution is not Gaussian; in fact the support $\mathcal M_*$ is not a linear subspace. We will call this distribution the \emph{normal collision distribution}.

The second parametrization in Eq. \eqref{eq:A-x-matrix} shows that the variety $\mathcal M$ contains the bundle of linear spaces 

\begin{equation*}
  (v,w) \mapsto (v,w,(I-\Pi)v + \Pi w, \Pi v + (I-\Pi) w) , \quad \pi \in \Pi(1) \ .
\end{equation*}

The distribution of

\begin{equation} \label{eq:Pi.distribution}
  \mathcal M \ni (v,w,\bar v, \bar w) \mapsto \Pi = \versof{v - \bar v} \otimes \versof{v - \bar v} \in \Pi(1)
\end{equation}
under the normal collision distribution is obtained from Eq. \eqref{eq:sigma.vs.Pi}. In fact $\Pi$ is the projector on the subspace generated by $\kappa - \sigma$ where $(v,w,\sigma) \mapsto \kappa = \versof{v-w}$ is uniformly distributed and independent from $\sigma$. Hence, Eq. \eqref{eq:kappa.minus.sigma} shows that $(v,w,\sigma) \mapsto \versof{\kappa-\sigma}$ is uniformly distributed on $\Sd$ so that the distribution in Eq. \eqref{eq:Pi.distribution} is the $\nu$ measure defined in Eq. \eqref{eq:nu-integral}. Conditionally to $\Pi$, the normal collision distribution is Gaussian with covariance

\begin{equation*}
  \begin{bmatrix}
    I & 0 & I - \Pi & \Pi \\
    0 & I & \Pi & I - \Pi \\
    I - \Pi & \Pi & I & 0 \\
    \Pi & I - \Pi & 0 & I
  \end{bmatrix} \ .
\end{equation*}
We can give the previous remarks a more probabilistic form as follows.
\begin{proposition}[Conditioning] \label{prop:conditioning}
Let $M$ be the density of the standard normal N$(0_3,I_3)$ 	and $g\colon \reals^{3} \times \reals^{3} \to \reals$ be an integrable function. It holds the following 
\begin{enumerate}
\item \label{prop:conditioning0}
If $(V,W) \sim M \otimes M$ and $\Pi \in \Pi(1)$, then 

\begin{equation*} 
\condexp {g(V,W)}{V+W,\collinv V W} = \condexp {g\left(A_\Pi(V,W)\right)}{V+W,\collinv V W}
\end{equation*}
\item \label{prop:conditioning1}
If $(V,W) \sim M \otimes M$, then 

\begin{equation*} 
  \integrald{\Sd}{g\left(\hat A_\sigma(V,W)\right)}{\mu(d\sigma)} = \condexp{g(V,W)}{V+W, \absoluteval{V}^2+\absoluteval{W}^2}.
\end{equation*}
\item \label{prop:conditioning1b} If $(V,W) \sim M \otimes M$, then 

\begin{equation*} 
  2 \integrald{\Sd}{g\left(A_{\Pi}(V,W)\right) \absoluteval{\Pi \left(\versof{V-W}\right)}} {\nu(d\Pi)} = \condexp{g(V,W)}{V+W, \absoluteval{V}^2+\absoluteval{W}^2}.
\end{equation*}
\item \label{prop:conditioning2} Assume $(V,W) \sim F$, $F \in \maxexp {M \otimes M}$, with $F(v,w) = f(v,w)M(v)M(w)$. Then

  \begin{equation*}
    \left(\integrald{\Sd}{f \circ \hat A_\sigma}{\mu(d\sigma)}\right) \cdot M \otimes M \in \maxexp {M \otimes M} \ ,
  \end{equation*}
and  

\begin{multline*} 
  \condexpat {f\cdot M \otimes M} {g(V,W)}{V+W,\absoluteval{V}^2+\absoluteval{W}^2} = \\ \frac{\integrald{\Sd}{g\left(\hat A_\sigma (V,W)\right)f\left( \hat A_\sigma (V,W)\right)}{\mu(d\sigma)}}{\integrald{\Sd}{f\left(\hat A_\sigma (V,W)\right)}{\mu(d\sigma)}} = \\ \frac{\integrald{\Sd}{g\left(A_\Pi (V,W)\right)f\left( A_\Pi (V,W)\right) \absoluteval{\Pi \left(\versof{V-W}\right)}} {\nu(d\Pi)}}{\integrald{\Sd}{f\left(A_\Pi (V,W)\right) \absoluteval{\Pi \left(\versof{V-W}\right)}} {\nu(d\Pi)}} \ .
\end{multline*}
\end{enumerate}
\end{proposition}

\begin{proof} 
\begin{enumerate}
\item We use $V+W = V_\Pi+W_\Pi$, $\collinv V W = \collinv {V_\Pi}{W_\Pi}$, $(V,W) \sim (V_\Pi,W_\Pi)$. For all bounded $h_1 \colon \reals^3 \to \reals$ and $h_2 \colon \reals \to \reals$, we have

\begin{multline*}
  \expectof{\condexp {g\left(A_\Pi(V,W)\right)}{V+W,\absoluteval{V}^2+\absoluteval{W}^2}h_1(V+W)h_2(\collinv V W)} = \\
\expectof{g(V_\Pi,W_\Pi)h_1(V+W)h_2(\collinv V W)} = \\ \expectof{g(V_\omega,W_\omega)h_1(V_\Pi+W_\Pi)h_2(\collinv {V_\omega} {W_\omega})} = \\ \expectof{g(V,W)h_1(V+W)h_2(\collinv V W)} = \\ \expectof{\condexp {g(V,W)}{V+W,\absoluteval{V}^2+\absoluteval{W}^2}h_1(V+W)h_2(\collinv V W)}.
\end{multline*}
\item For a generic integrable $h \colon \reals^3 \times \reals^3 \to \reals$ we have

  \begin{multline*}
    \expectof{h(V,W)} = \\ \expectof{h\left(\frac{V+W}2+\frac{\absoluteval{V-W}}2 \versof{V-W},\frac{V+W}2 - \frac{\absoluteval{V-W}}2 \versof{V-W}\right)} = \\ \int_{\Sd} \expectof{h\left(\frac{V+W}2+\frac{\absoluteval{V-W}}2 \sigma,\frac{V+W}2 - \frac{\absoluteval{V-W}}2 \sigma \right)} \ \mu(d\sigma) = \\ \int_{\Sd} \expectof{h(\hat A_\sigma(V,W))} \ \mu(d\sigma) \ , 
  \end{multline*}
because $\versof{V-W} \sim \mu$, and $(V+W)$, $\absoluteval{V-W}$, $\versof{V-W}$ are independent.

The random variable

\begin{equation*}
  \integrald{\Sd}{g\circ \hat A_\sigma (V,W)}{\mu(d\sigma)} = \integrald{\Sd}{g\left(\frac{V+W}2 + \frac{\absoluteval{V-W}}2 \sigma,\frac{V+W}2 - \frac{\absoluteval{V-W}}2 \sigma \right)}{\mu(d\sigma)}
\end{equation*}
is a function of the collision invariants i.e., it is of the form $\tilde g(V+W,\collinv V W)$. For all bounded $h_1 \colon \reals^3 \to \reals$ and $h_2 \colon \reals \to \reals$, we apply the previous computation to $h = g h_1 h_2$ to get

\begin{multline*}
\expectof{\left(\integrald{\Sd}{g\circ \hat A_\sigma (V,W)}{\mu(d\sigma)}\right) h_1(V+W)h_2(\collinv V W)}
= \\ \integrald{\Sd}{\expectof{g\circ \hat A_\sigma (V,W) h_1(V+W)h_2(\collinv {V} {W})}} {\mu(d\sigma)} = \\ \integrald{\Sd}{\expectof{g\circ \hat A_\sigma (V,W) h_1(V_\sigma+W_\sigma)h_2(\collinv {V_\sigma} {W_\sigma})}} {\mu(d\sigma)} 
= \\ \expectof{g(V,W)h_1(V+W)h_2(\collinv V W)}.
\end{multline*}

\item We use Item \ref{prop:conditioning1} and the equality $\hat A_\sigma(v,w) = A_\Pi(v,w)$ when $\Pi = \versof{\kappa-\sigma} \otimes \versof{\kappa-\sigma}$ and $\kappa=\versof{v-w}$ to write

\begin{equation*}
  \integrald{\Sd}{g\left(A_{ {\versof{V-W}-\sigma} \versof{\otimes}  {\versof{V-W}-\sigma}}(V,W)\right)}{\mu(d\sigma)} = \condexp{g(V,W)}{V+W, \absoluteval{V}^2+\absoluteval{W}^2}
\end{equation*}
where, for given vectors $u,v \in \reals^{3}$, $u,v \neq 0$ we simply denote $u \,\versof{\otimes}\, v=\versof{u} \otimes \versof{v}.$
From Eq. \eqref{eq:integral-sigma-to-Pi}, the left-end-side can be rewritten as
an integral with respect to $\omega \in \Sd$,

\begin{equation*}
  \integrald{\Sd}{g\left(A_{\omega\otimes \omega}(V,W)\right) 2\absoluteval{\versof{V-W} \cdot \omega}}{\mu(d\omega)} =  \integrald{\Sd}{g\left(A_{ {\versof{V-W}-\sigma} \versof{\otimes} {\versof{V-W}-\sigma}}(V,W)\right)}{\mu(d\sigma)} \ .
\end{equation*}
If $\Pi = \omega \otimes \omega$, then $\absoluteval{\kappa \cdot \omega} = \absoluteval{\omega \otimes \omega \kappa} = \absoluteval{\Pi \kappa}$. Using that together with the definition of the measure $\nu$ on $\Pi(1)$ in Eq. \eqref{eq:nu-integral}, we have the result:

\begin{multline*}
  \integrald{\Sd}{g\left(A_{\omega\otimes \omega}(V,W)\right) 2\absoluteval{\versof{V-W} \cdot \omega}}{\mu(d\omega)} = \\  \integrald{\Sd}{g\left(A_{\omega\otimes \omega}(V,W)\right) 2\absoluteval{(\omega \otimes \omega) \versof{V-W}}}{\mu(d\omega)} = \\ 2 \integrald{\Pi(1)}{g\left(A_{\Pi}(V,W)\right) \absoluteval{\Pi(\versof{V-W})}}{\nu(d\Pi)}.
\end{multline*}
\item We use Th. \ref{prop:maxexp-pormanteau}. If $F \in \maxexp {M \otimes M}$, then

\begin{equation*}
  F = \euler^{U - K_{M \otimes M}(U)}   M \otimes M, \quad U \in \sdomain {M \otimes M} \ ,
\end{equation*}
and there exists a neighborhood $I$ of $[0,1]$, where the one dimensional exponential family

\begin{equation*}
  F_t = \euler^{tU - K_0(tM)}   M \otimes M, \quad t \in I,
\end{equation*}
exists. The random variable

\begin{equation}
 \hat f  = \integrald{\Sd}{f \circ \hat A_\sigma}{\mu(d\sigma)} 
\end{equation}
is a positive probability density with respect to $M \otimes M$ because it is the conditional expectation in $M \otimes M$ of the positive density $f$. 

In order to show that $\expectat {M \otimes M}{(\hat f)^t} < + \infty$ for $t \in I$, it is enough to consider the convex cases, $t < 0$ and $t> 1$, because otherwise $\expectat {M \otimes M}{(\hat f)^t} \le 1$. We have

\begin{equation*}
  \integrald {\Sd} {f \circ \hat A_\sigma} {\mu(d\sigma)} = \integrald {\Sd} {\euler^{U \circ \hat A_\sigma - K_{M \otimes M}(U)}} {\mu(d\sigma)} \ ,
\end{equation*}
so that in the convex cases:

\begin{multline*}
  \expectat {M \otimes M} {\left(\integrald {\Sd} {f \circ \hat A_\sigma} {\mu(d\sigma)}\right)^t} = \expectat {M \otimes M} {\left(\integrald {\Sd} {\euler^{U \circ \hat A_\sigma - K_{M \otimes M}(U)}} {\mu(d\sigma)}\right)^t} \le \\ \expectat {M \otimes M} {\integrald {\Sd} {\euler^{tU \circ \hat A_\sigma - tK_{M \otimes M}(U)}} {\mu(d\sigma)}} = \expectat {M \otimes M} {\euler^{tU  - tK_{M \otimes M}(U)}} = \\ \euler^{K_{M \otimes M}(tU) - t K_{M \otimes M}(U)} < +\infty, \quad t \in I \setminus [0,1].
\end{multline*}
To conclude, use Bayes' formula for conditional expectation, 

\begin{equation*}
  \condexpat {f  \cdot M \otimes M} {A} {\mathcal B} = \frac {\condexpat {M \otimes M} {Af} {\mathcal B}} {\condexpat {M \otimes M} {f } {\mathcal B}} \ .
\end{equation*}
to the expressions of conditional expectation in Item \ref{prop:conditioning1} and Item \ref{prop:conditioning1b} above. 
\end{enumerate}
\end{proof}

\begin{remark}
In the last Item, we compute a conditional expectation of a density $f$, that is 

\begin{equation*}
  \condexp{f (V,W)}{V+W, \absoluteval{V}^2+\absoluteval{W}^2} = \hat f (V+W, \absoluteval{V}^2+\absoluteval{W}^2).
\end{equation*}
The random variable $\hat f  \colon \reals^3 \times \reals_+ \to \reals$ is the density of the image of $F \ dvdw$ with respect to the image of $M(v)M(w)\ dvdw$. 
\end{remark}

\subsection{Interactions}\label{sec:interaction}

We introduce here the crucial role played by microscopic interaction in the definition of the Boltzmann collision operator. In the physics literature, such interaction are referred to as the  kinetic collision kernel and takes into account the intermolecular forces suffered by particles during a collision \cite{MR1942465}. Before defining formally what we mean by interaction, we first observe that, if  $M$ is the Maxwell density on $\reals^3$ and $f,g\smile M$ then
$$f \otimes g \smile M \otimes M$$
where $M \otimes M$ is the standard normal density on $\reals^6$ and $f\otimes g$ is a density on $\reals^6$. Indeed, one has 

  \begin{align*}
    f(v) &= \euler^{U(v)-K_M(U)} M(v), \quad U \in \sdomain M \ , \\
    g(w) &= \euler^{V(w)-K_M(V)} M(w), \quad V \in \sdomain M \ . 
  \end{align*}
It follows that the product density has the form 

\begin{equation*}
      f(v)g(w) = \euler^{U(v)+V(w)-K_M(U)-K_M(V)} M(v)M(w), \quad U\oplus V \in \sdomain M\otimes \sdomain M \subset \sdomain{M\otimes M} \ ,
\end{equation*}
which implies $f \otimes g \smile M \otimes M$, with $K_{M\otimes M}(U\oplus V) = K_M(U) + K_M(V)$.

\begin{definition} With the previous notations, we say that $b \colon \reals^3 \times \reals^3 \to \reals^{+}$ is an \emph{interaction} on $\maxexp M \otimes \maxexp M$, if $\expectat {f \otimes g} b < + \infty$, so that  

\begin{equation*}
  \frac b {\expectat {f \otimes g} b}  f\otimes g \colon (v,w) \in \reals^{3}\times \reals^{3} \mapsto \frac {b(v,w)}{\expectat {f \otimes g} b}f(v)g(w)
\end{equation*}
is a density. 
\end{definition}
Sometimes, we make the abuse of notation by writing $\expectat {b \cdot f \otimes g} \cdot$, where the obvious normalization is not written down. 
As can been seen, here interactions indicate only a class of suitable weight functions $b$ for which $b(v,w)f(v)g(w)$ is still (up to normalisation) a density. This is in accordance with the usual role played by the kinetic collision kernel (see \cite{MR1942465}).

According to the Portmanteau  Theorem \ref{prop:maxexp-pormanteau}, it holds $\frac b {\expectat {f \otimes g} b} \cdot f\otimes g \smile M\otimes M$ if and only if $\frac b {\expectat {f \otimes g} b} \cdot f\otimes g \smile f \otimes g$, which, in turn, is equivalent to $\expectat {f\otimes g} {b^{1+\epsilon}}, \expectat{b\cdot f\otimes g}{b^{-1-\epsilon}} < \infty$ for some $\epsilon>0$.
\begin{proposition}
\label{prop:lodslemma}
Let $b \colon \reals^3 \times \reals^3 \to \reals^{+}$ be such that for some real $A \in \reals$ and positive $B, C, \lambda \in \reals_>$, it holds

\begin{equation*}
  C \absoluteval{u-v}^\lambda \le b(u,v) \le A + B\absoluteval{u-v}^2 \ .
\end{equation*}
Then the $b$ is an interaction on $\maxexp M \times \maxexp M$ and for all $f, g \smile M$ the following holds.
\begin{enumerate}
\item $\frac b {\expectat {f \otimes g} b} \cdot f\otimes g \smile M\otimes M$.  
\item Assume moreover that the interaction $b$ is a function of the invariants only

  \begin{equation*}
    b(v,w) = \hat b(v+w,\collinv v w) \ .
  \end{equation*}
It follows 

\begin{equation*}
  \condexpat {M \otimes M} { b \cdot f \otimes g} {V+W,\collinv V W} = b(V,W) \integrald {\Sd} {f \otimes g \circ \hat A_\sigma(V,W)} {\mu(d\sigma)} \ .
\end{equation*}
and moreover a sufficiency relation holds i.e., for all integrable $F\colon \reals^3 \times \reals^3 \to \reals$, it holds

\begin{equation*}
  \condexpat {b \cdot f \otimes g}{F(V,W)}{V+W,\collinv V W} = \condexpat {f \otimes g}{F(V,W)} {V+W,\collinv V W} \ .
\end{equation*}
\end{enumerate}
\end{proposition}
\begin{proof} 1. We can assume $\expectat {f \otimes g} b = 1$. 
As $b(u,v) \le A + B \absoluteval{u-v}^2 \le A + 2B \absoluteval{(u,v)}^2$, we have $b \in L^\Phi(M\otimes M) = L^\Phi(f\otimes g)$, and hence $b \in L^{1+\epsilon}(f\otimes g)$ for \emph{all} $\epsilon > 0$. For the second inequality we use the Hardy-Littlewood-Sobolev inequality \cite[Th. 4.3]{MR1817225}. We have

  \begin{equation*}
    \expectat{b\cdot f\otimes g}{b^{-1-\epsilon}} = \iint \frac{f(u)g(v)}{b(u,v)^\epsilon}\ dudv \le C^{-1} \iint \frac{f(u)g(v)}{\absoluteval{u-v}^{\epsilon\lambda}}\ dudv. 
  \end{equation*}
If $\frac 1\beta + \frac {\epsilon \lambda}3 + \frac 1\alpha = 2$, by the H-L-S inequality the last integral is bounded by a constant times $\normat \alpha f \normat \beta g$. From $f,g \smile M$ we get that $\normat \alpha f \normat \beta g$ is finite for $\alpha,\beta$ in a right neighborhood of 1. There exists $\epsilon = \frac{3}\lambda \left(2 - \frac1\alpha-\frac1\beta\right) > 0$ satisfying all conditions.

2. It is a special case of the Conditioning Theorem \ref{prop:conditioning}.
\end{proof}
Let us discuss the differentiability of the operations we have just introduced.

\begin{proposition} \ 
  \begin{enumerate}
  \item The product mapping $\maxexp M \ni f \mapsto f \otimes f$ is a differentiable map into $\maxexp {M\otimes M}$ with tangent mapping given for any vector field $X \in T \maxexp M$ by

    \begin{equation*}
      T_f \maxexp M \ni X_f \mapsto X_f \oplus X_f \in T_{f\otimes f} \maxexp {M\otimes M} \ .
    \end{equation*}
  \item Let $b$ be an interaction on $\maxexp M \otimes \maxexp M$. If the mapping 

    \begin{equation*}
f \mapsto \frac b {\expectat {f \otimes f} b} \cdot f \otimes f      
    \end{equation*}
is defined on $\maxexp M$ with values in $\maxexp {M \otimes M}$, then it is differentiable with tangent mapping given for all vector field $X \in T \maxexp M$ by

    \begin{equation*}
      X_f \mapsto X_f \oplus X_f - \expectat {b \cdot f \otimes f} {X_f \oplus X_f} \ .
    \end{equation*}
  \end{enumerate}
\end{proposition}

\begin{proof} 
\begin{enumerate}
\item 
We have already proved that $f \smile M$ implies $f \otimes f \smile M \otimes M$ and that the mapping in the charts centered at $M$ and $M\otimes M$ respectively is represented as $\sdomain M \ni U \mapsto U + U \in \sdomain {M \otimes M}$. The differential of the linear map is again $ T_M \maxexp M \ni V \mapsto V \oplus V \in T_{M\otimes M} \maxexp {M\otimes M}$. The transport commutes with the $\oplus$ operation, 

    \begin{equation*}
      \transport{M \otimes M}{f \otimes f} \left(V \oplus V\right) = \left(\transport M f V\right) \oplus \left(\transport M f V\right) \ ,
    \end{equation*}
and the result follows.
\item
Let $U$ be the coordinate of $f$ at $M$. By assumption, we have

  \begin{equation*}
    \frac b {\expectat {f \otimes f} b} \cdot f \otimes f \smile \frac b {\expectat {M \otimes M} b} \cdot M \otimes M
  \end{equation*}
and

\begin{equation*}
  \frac{\frac b {\expectat {f \otimes f} b} \cdot f \otimes f}{\frac b {\expectat {M \otimes M} b} \cdot M \otimes M} = \frac{\expectat {M \otimes M} b}{\expectat {f \otimes f} b} \expof{U \oplus U - 2 K_M(U)} \ , 
\end{equation*}
so that the coordinate of $\frac b {\expectat {f \otimes f} b} \cdot f \otimes f$ in the chart centered at $\frac b {\expectat {M \otimes M} b} \cdot M \otimes M$ is 

\begin{equation*}
  U \oplus U - \expectat {b \cdot M \otimes M}{U \oplus U} = \transport {M \otimes M}{b \cdot M \otimes M} (U \oplus U).
\end{equation*}
The expression if linear and so is the expression of the tangent map

\begin{equation*}
  V \mapsto V \oplus V - \expectat {b \cdot M \otimes M} {V \oplus V} \ .
\end{equation*}
In conclusion, for each vector field $X$ of $T \maxexp M$, at $f$ we have $V = \transport f M X_f$ and $V \oplus V = \transport {f \oplus f} {M \oplus M} ( X_f \oplus X_f)$, hence the action on $T_f \maxexp M$ is as stated.
\end{enumerate}
\end{proof}

Assume $f \smile M$ and let $b$ be an interaction on $f \otimes f$ which depends on the invariants only and such that $b \cdot f \otimes f \smile M \otimes M$. For each random variable $g \in L^{\cosh-1}(M) = L^{\cosh-1}(f)$, define $\widebar g = \frac12 \ g \oplus g$, which belongs to $L^{\cosh-1}(M \otimes M) = L^{\cosh-1} (f \otimes f)$. Define the operator

\begin{equation*}
  A \colon L^{\cosh-1}(M) \to L^{\cosh -1}(M \otimes M) \ ,
\end{equation*}
by

 \begin{equation*}
    Ag(v,w) = \integrald{\Sd}{\widebar g\left(\hat A_\sigma(v,w)\right)}{\mu(d\sigma)} - \widebar g(V,W) \ .
  \end{equation*}
As constant random variables are in the kernel of the operator $A$, we assume $\expectat f g = 0$.

\subsection{Maxwell-Boltzmann and Boltzmann operator}
\label{sec:maxwell-boltzmann}

As explained by  C. Villani \cite[I.2.3]{MR1942465}, Maxwell obtained a weak form of Boltzmann operator before Boltzmann himself.  We rephrase in geometric-probabilistic language such a Maxwell's weak form by expanding and rigorously proving what was hinted to in \cite{MR3130268}.

Let $b \colon \reals^3 \times \reals^3 \to \reals^{+}$ be an interaction on $\maxexp M \times \maxexp M$ which depends on the invariants only and such that $b \cdot f \otimes f \smile M \otimes M$ if $f \smile M$, cf. Prop. \ref{prop:lodslemma}. We shall call such a $b$ a \emph{proper interaction}.

For each random variable $g \in L^{\cosh-1}(M) = L^{\cosh-1}(f)$, define $\widebar g = \frac12 \ g \oplus g$ by

$$\widebar{g}(v,w)=\frac{1}{2}\left(g(v)+g(w)\right)$$
 which it is easily shown to belong to $L^{\cosh-1}(M \otimes M) = L^{\cosh-1} (f \otimes f)$. The mapping $g \mapsto \widebar g$ is a version of the conditional expectation $\condexpat {M \otimes M}{g}{\mathcal S}$, where $\mathcal S$ is the $\sigma$-algebra generated by symmetric random variables. 

We define the operator 

\begin{equation*}
  A \colon L^{\cosh-1}(M) \to L^{\cosh -1}(M \otimes M) \ ,
\end{equation*}
by

 \begin{equation*}
    Ag(v,w) = \integrald{\Sd}{\widebar g\left(\hat A_\sigma(v,w)\right)}{\mu(d\sigma)} - \widebar g(V,W) \ ,
  \end{equation*}
which is a version of 

\begin{equation*}
  Ag = \condexpat {M \otimes M} {\condexpat {M \otimes M} g {\mathcal S}} {\mathcal I} - \condexpat {M \otimes M}{g}{\mathcal S} = \condexpat {M \otimes M} {g} {\mathcal I} - \condexpat {M \otimes M}{g}{\mathcal S} \ ,
\end{equation*}
where $\mathcal I \subset \mathcal S$ is the $\sigma$-algebra generated by the \emph{collision invariants} $(v,w) \mapsto (v+w,\collinv v w)$.

As constant random variables are in the kernel of the operator $A$, we assume $\expectat f g = 0$. Analogously, as the kernel of the operator contains all symmetric random variables, we could always assume that $g$ is anti-symmetric.

The nonlinear operator $f \mapsto \expectat {b \cdot f \otimes f}{Ag}$ is the \emph{Maxwell's weak form} of the Boltzmann operator, $g$ being a test function. 

\begin{proposition}
\label{prop:maxwell}
Given a proper interaction $b$ and a density $f \in \maxexp M$, the linear map 

\begin{equation*}
  L_0^{\cosh-1}(M) \ni g \mapsto \expectat {b \cdot f \otimes f} {Ag}
\end{equation*}
is continuous.
\begin{enumerate}
\item \label{item:maxwell1}
It can be represented in the duality $L^{(\cosh-1)_*}(f) \times L^{\cosh-1}(f)$ by

 \begin{align*}
\expectat {b \cdot f \otimes f} {Ag} &= \expectat {M \otimes M}{b \left(\condexpat {M \otimes M} {\frac f M \otimes \frac f M} {\mathcal I} - \frac f M \otimes \frac f M\right) g}\\
   &= \scalarat f {Q(f)/f} g \ ,
  \end{align*}
where $Q$ is the \emph{Boltzmann operator} with interaction $b$,

\begin{equation*}
  Q(f)( v)=\int_{\reals^3}\int_{\Sd} \left(f \otimes f \left(\hat A_\sigma(v,w)\right)-f(v)f(w) \right) b(v,w) \ \mu(d\sigma)\ dw 
\end{equation*}
\item \label{item:maxwell2}
Especially, if $f = \euler^{U - K_M(U)} M$ and we take $g = \logof{\frac f M}$, then

  \begin{equation*}
    \expectat {b \cdot f\otimes f} {A\logof{\frac{f}{M}}} = \scalarat f {Q(f)/f} U \ .
  \end{equation*}
\end{enumerate}
\end{proposition}

\begin{proof}
  The continuity follows from the Portmanteau theorem and the continuity of the conditional expectation. Item \ref{item:maxwell1} follows from the projection properties of the conditional expectation and from general properties of Orlicz spaces. Item \ref{item:maxwell2} is a special case of the previous one.
\end{proof}
It follows from the previous theorem and from the discussion in \cite[Prop. 10]{MR3130268} that $f \mapsto Q(f)/f$ is a vector field in the cotangent bundle $\mTof M$ and its flow

\begin{equation*}
  \frac D {dt} f_t = \frac {\dot f_t}{f_t} = \frac {Q(f_t)}{f_t} 
\end{equation*}
is equivalent to the standard Boltzmann equation $\dot f_t = Q(t)$. We do not discuss in this paper the implications of that presentation of the Boltzmann equation to the existence properties. We turn our attention to the comparison of the Boltzmann field to the gradient field of the entropy. 

\subsection{Entropy generation} 
\label{sec:entgen}

If $t \mapsto p(t)$ is a curve in $\maxexp M$,  the entropy of $p(t)$ is defined for all $t$ and the variation of the entropy along the curve is computed as

\begin{multline*}
  \derivby t H(p_t) = - \int \derivby t \left( p(v;t) \log p(v;t) \right) \ dv =
\\ - \int \left( \log p(v;t) - 1 \right) \derivby t p(v;t) \ dv = - \int \log p(v;t) \derivby t p(v;t) \ dv \ .
\end{multline*}
In our setup, the computation takes the following form. Let $t \mapsto p_t$ be a differentiable curve in $\maxexp M$ with velocity $t \mapsto D p_t \in \mBspace {p_t}$. As the gradient $\nabla H$ is given at $p$ by $(\nabla H)(p) = -(\log p + H(p)) \in B_p$, we have 

\begin{equation*}
   \derivby t H(p(t)) = - \scalarat {p(t)} {\log p(t) + H(p(t))} {D p(t)} \ . 
\end{equation*}
In particular, if $Dp(t) = \dot p(v;t)/p(v;t)$ we recover the previous computation:

\begin{align*}
  \derivby t H(p(t)) &= 
- \int \left(\log p(t) + H(p(t))\right) \frac {\dot p(v;t)}{p(v;t)} \ p(v;t) dv \\ &= - \int \log p(v;t) \derivby t p(v;t) \ dv \ . 
\end{align*}
Assume now that $(p,t) \mapsto \gamma(t;p)$ is the flow of a vector field $F \colon \maxexp M \to \mTof M$. Then

\begin{align*}
  \derivby t H(\gamma(t;p)) &= - \scalarat {\gamma(t;p)} {\log \gamma(t;p) + H(\gamma(t;p))} {D \gamma(t;p)} \\ &= -\scalarat {\gamma(t;p)} {\log \gamma(t;p) + H(\gamma(t;p))} {F(\gamma(t;p))}, 
\end{align*}
that is, for each $p \in \maxexp M$, the \emph{entropy production at $p$ along the vector field $F$} is

\begin{multline*}
  (\nabla _F H)(p) = - \scalarat p {\log p + H(p)}{F(p)} = \\ - \int \left(\log p(v) + H(p)\right) F(v,p) \ p(x)dx = - \int \log p(v)  F(v,p) \ p(v) dv \ .
\end{multline*}
In particular, if the vector field $F$ is the Boltzmann vector field, $F(f) = Q(f)/f$, we have that for each $f \in \maxexp M$ the Boltzmann's entropy production is 

\begin{multline*}
  \entropyprodat f = (\nabla _{Q(f)/f} H)(f) = - \scalarat f {	\log f + H(f)}{Q(f)/f} \\ = - \expectat {b \cdot f \otimes f} {A \log f} = - \frac12 \expectat {b \cdot f \otimes f} {A \left(\log f \ominus \log f\right)} \ ,
\end{multline*}
where

\begin{equation*}
  \left(\log f \ominus \log f\right)(v,w) = \log f(v) - \log f(w) = \log \frac {f(v)}{f(w)} \ . 
\end{equation*}
We recover a well-known formula for the entropy production of the Boltzmann operator \cite{MR1942465}. We now proceed to compute the covariant derivative of the entropy production.
\begin{proposition}\ 

Let $X$ be a vector field of $\eTof M$ and let $F$ be a vector field of $\mTof M$.
\begin{enumerate}
\item The Hessian of the entropy (in the exponential connection) is

  \begin{equation*}
    D_X\nabla H(f) = - X.
  \end{equation*}
\item 
The covariant derivative of the entropy production along $F$ is

\begin{equation*}
D_X \entropyprod = -\scalarof X F + \scalarof {\nabla H}{D_X F}. 
\end{equation*}
\end{enumerate}
\end{proposition}
\begin{proof}
We note that the entropy production along a vector field $F$, $\entropyprod = (\nabla _F H)(p) = \scalarof {\nabla H}{F}$, is a function of the duality coupling of $\eTof M \times \mTof M$, so that we can apply Prop. \ref{prop:der.duality} to compute its covariant derivative along $X$ as 

\begin{equation*}
  D_X \entropyprod = \scalarof {D_X \nabla H} F + \scalarof {\nabla H} {D_X F}.  \end{equation*} 
The first term at $p$ is

\begin{equation*}
  \scalarof {D_X \nabla H} F (p) = \scalarat p {D_X \nabla H(p)} {F(p)}.
\end{equation*}
Let us compute $D_X \nabla H(p)$, which is the Hessian of the entropy in the exponential connection. First, we compute the expression of $\nabla H(q) = - \left(\log q + H(q)\right) \in \eBspace q$, $q \in \maxexp M$ in the chart centered at $p$. We have

\begin{equation*}
- \log q = K_p(U) - U - \log p, \quad U \in \sdomain p, q = e_p(U),  
\end{equation*}
and

\begin{equation*}
 H_p(U) = - \expectat q {\log q} = K_p(U) - dK_p(U)[U] + dK_p(U)[\nabla H(p)] + H(p)
\end{equation*}
so that

\begin{equation*}
  (\nabla H)\circ e_p(U) = -U + dK_p(U)[U] + \nabla H(p) - dK_p(U)[\nabla H(p)] 
\end{equation*}
and, finally, the expression of $\nabla H$ in the chart centered at $p$ is

\begin{equation*}
  (\nabla H)_p(U) = \etransport {e_p(U)} p (\nabla H)\circ e_p(U) = \nabla H(p) - U.
\end{equation*}
Note that this function is affine, and its derivative in the direction $X(p)$ is $d (\nabla H)_p(U)[X(p)] = -X(p)$. It follows that $D_X\nabla H = -X$.
\end{proof}
The application of this computation to the Boltzmann field i.e. $F(f) = Q(f)/f$ requires the existence of the covariant derivative of the Boltzmann operator. We leave this discussion as a research plan.

\section{Weighted Orlicz-Sobolev model space}
\label{sec:orlicz-sobolev}

We show in this section that the Information Geometry formalism described in Sections \ref{sec:model-spaces} and \ref{exponentialmanifold} is robust enough to allow to take into account differential operators (e.g., the classical Laplacian). This yields naturally to the introduction of weighted Orlicz-Sobolev spaces. While the case ``without derivative'' studied in Section \ref{exponentialmanifold} was well-suited for the study of the fine properties of the Kullback-Leibler divergence, we illustrate in Sec. \ref{sec:hyvarinen} our use of Orlicz-Sobolev spaces with the fine study of the  Hyv\"arinen divergence. This is a special type of divergence between densities that involves an $L^2$-distance between gradients of densities \cite{MR2249836} which has multiple applications. In particular, it is related with the so called Fisher information as it is defined for example in \cite[p. 49]{MR2409050}, which has deep connections with Boltzmann equation, see \cite{MR1700142}. However the name Fisher information should not be used in a statistics context where   it rather refers to the expression in coordinates of the metric of statistical models considered as pseudo-Riemannian manifolds e.g., \cite{MR1800071}.

We introduce the Orlicz-Sobolev spaces with weight $M$, Maxwell density on $\reals^n$, 

\begin{align}
   W_{\cosh-1}^1(M) = \setof{f \in \Lexp M}{\partial_j f \in \Lexp M, j = 1, \dots, n}   \label{eq:OrSobExp},\\
  W_{(\cosh-1)_*}^1(M) = \setof{f \in \LlogL M}{\partial_j f \in \LlogL M, j = 1, \dots, n}   \label{eq:OrSobLog},
\end{align}
where $\partial_j$ is the derivative in the sense of distributions. They are both Banach spaces, see \cite[\S 10]{MR724434}. (The classical Adams's treatise \cite[Ch 8]{MR2424078} has Orlicz-Sobolev spaces, but does not consider the case of a weight. The product functions $(u,x) \mapsto (\cosh-1)(u)M(x)$ and $(u,x) \mapsto (\cosh-1)_*(u)M(x)$ are $\phi$-functions according the Musielak's definition.) The norm on $W_{\cosh-1}^{1}(M)$ is

\begin{equation}
  \label{eq:OrSob-norm}
  \normat {W^{1}_{\cosh-1}(M)} f = \normat {\Lexp M} f + \sum_{j=1}^n \normat {\Lexp M} {\partial_j f},
\end{equation}
and similarly for $W^{1}_{(\cosh-1)_{*}}(M)$. One begins with a first technical result in order to relate such spaces with statistical exponential families:

\begin{proposition}\label{prop:feuler} Given $u \in \Sspace M \cap W^{1}_{\cosh-1}(M)$ and $f\in W^{1}_{\cosh-1}(M)$, one has

$$f\euler^{u-K_M(u)} \in W^{1}_{(\cosh-1)_{*}}(M).$$
\end{proposition}

\begin{proof} For simplicity, set $G=\euler^{u-K_M(u)}.$ One knows from the Portmanteau Theorem \ref{prop:maxexp-pormanteau} that $G\,M \in \maxexp M$ and therefore, there exists $\varepsilon > 0$ such that $G \in L^{1+\varepsilon}(M).$ Let us prove that $fG \in \LlogL M.$ First of all, since $L^{1+\varepsilon}(M) \subset \LlogL M$ and $f \in \Lexp M$ one has $fG \in L^{1}(M)$. Moreover, for any $x \in \reals^{n}$, according to classical Young's inequality

$$\ f(x)G(x) \leq \frac{1}{p}|f(x)|^{p} + \frac{1}{q}|G(x)|^{q} \qquad \forall p > 1, \:\:\frac{1}{p}+\frac{1}{q}=1.$$
Since $\Phi_{*}$ is increasing and convex

$$\Phi_{*}(f(x)G(x)) \leq \frac{1}{p}\Phi_{*}(\absoluteval{f(x)}^{p}) + \frac{1}{q} \Phi_{*}\left( G(x)^{q}\right) \qquad x \in \reals^{n}.$$
Now, since $f \in L^{\Phi}(M)$, one has $|f|^{p} \in \LlogL M$ for all $p > 1$, i.e. $\Phi_{*}(\absoluteval{f}^{p}) \in L^{1}(M)$ for all $p>1.$ 
Choosing then $1<q< 1+\varepsilon$ one has $G^{q} \in L^{\frac{1+\varepsilon}{q}}(M) \subset \LlogL M$   so that $\Phi_{*}(G^{q}) \in L^{1}(M)$. This proves that $\Phi_{*}(fG) \in L^{1}(M)$ i.e.

$$fG\in \LlogL M.$$
In the same way, since $f \in W^{1}_{\cosh-1}(M)$ one also has

$$G\partial_{j}f \in  \LlogL M \qquad \forall j=1,\ldots,n.$$
Moreover, $u \in W^{1}_{\cosh-1}(M)$ so that, for any $j=1,\ldots,n$, $G\partial_{j}u \in L^{r}(M)$ for any $r >1$ and therefore $\Phi_{*}(\absoluteval{G\partial_{j} u}^{p}) \in L^{1}(M)$ for any $p > 1$. Repeating the above argument we get therefore

$$fG\partial_{j}u  \in \LlogL M \qquad \forall j=1,\ldots,n.$$
Since $\partial_{j}(fG)=G\partial_{j}f + Gf\partial_{j}u$ a.e., one gets $\partial_{j}(fG)\in \LlogL M$ for any $j=1,\ldots,n$ which proves the result. 
\end{proof}

\begin{remark}
\label{rem:wheref}
 As a particular case of the above Proposition, if $u \in \Sspace M \cap W^{1}_{\cosh-1}(M)$ then

$$\euler^{u-K_{M}(u)} \in W^{1}_{(\cosh-1)_{*}}(M) \qquad \text{ with } \quad \bnabla \euler^{u-K_M(u)} = \bnabla u\, \euler^{u-K_M(u)}.$$
\end{remark}

The Orlicz-Sobolev spaces $W^{1}_{\cosh-1}(M)$ and $W^{1}_{(\cosh-1)_{*}}(M)$, as defined in Eq. \eqref{eq:OrSobExp} and \eqref{eq:OrSobExp} respectively, are instances of Gaussian spaces of random variables and they inherit from the corresponding Orlicz spaces a duality form. In this duality the adjoint of the partial derivative has a special form coming from the form of the weight $M$ see e.g. \cite[Ch. V]{MR1335234}. We have the following:

\begin{proposition}\ 
  \begin{enumerate}
  \item Let $f \in C_0^\infty(\reals^n)$ and $g \in W_{\cosh-1}^1(M)$. Then

    \begin{equation}\label{eq:stein}
      \scalarat M f{\partial_j g} = \scalarat M {X_j f - \partial_j f}{g}
    \end{equation}
where $X_{j}$ is the mutliplication operator by the $j$-th coordinate $x_{j}$. 
\item If $f \in W_{(\cosh-1)_*}^{1}(M)$, then $X_jf \in \LlogL M$. More precisely, there exists $C >0$ such that

\begin{equation}\label{eq:Xjf}\normat {\LlogL M} {X_{j}f} \leq C \normat {W^{1}_{(\cosh-1)_{*}}(M)} f \qquad \forall f \in W^{1}_{(\cosh-1)_{*}(M)}.\end{equation}
\item If $f \in W_{(\cosh-1)_*}^{1}(M)$ and $g \in W_{\cosh-1}^1(M)$, then \eqref{eq:stein} holds.
  \end{enumerate}
  
\end{proposition}

\begin{proof}
  \begin{enumerate}
  \item As $fM \in C_0^\infty(\reals^{n})$, we have by definition of distributional derivative,

    \begin{multline*}
      \scalarat M f{\partial_j g} = \int_{\reals^{n}} f(x) \partial_j g(x) M(x) \ dx = - \int_{\reals^{n}}
g(x) \partial_j \left( f(x) M(x) \right)\ dx = \\ \int_{\reals^{n}} \left(x_j f(x) - \partial_j f(x) \right) M(x) \  dx = \scalarat M {X_j f - \partial_j f}{g}.
\end{multline*}
\item Let us observe first that, according to Holder's inequality

$$\expectat M {\absoluteval{X_jf}} \le 2 \normat {\Lexp M} {X_j} \normat {\LlogL M} {f} < \infty,$$ i.e. $X_jf \in L^1(M)$. Since $\Phi_{*}=(\cosh-1)_{*}$ enjoys the so-called $\Delta_{2}$-condition, to prove the stronger result $X_jf \in \LlogL M$, it is enough to show that $\expectat M {\Phi_*(X_jf)} < \infty$. 
First of all, using the tensorization property of the Gaussian measure, i.e. the fact that  $M(x)=M_{1}(x_{1})\ldots\,M_{1}(x_{n})$ for any $x=(x_{1},\ldots,x_{n}) \in \reals^{n}$ where $M_{1}$ stands for the one-dimensional standard Gaussian, we claim that it is enough to prove the result for $n=1$. Indeed, given $f \in W^{1}_{(\cosh-1)_{*}}(M)$ and $x_{j} \in \reals$ $(j=1,\ldots,n)$, any $x=(x_{1},\ldots,x_{n}) \in \reals^{n}$ can be identified with $x=(x_{j},\underline{x})$ with $\underline{x} \in \reals^{n-1}$ and $x_{j}f(x)=x_{j}F_{\underline{x}}(x_{j})$ where $F_{\underline{x}}(y)=f(y,\underline{x})$ for any $\underline{x} \in \reals^{n-1}$, $y \in \reals$. We also set $M_{n-1}(\underline{x})=M(x)/M_{1}(x_{j})$. Then, for a.e. 
$\underline{x} \in \reals^{n-1}$, $F_{\underline{x}} \in W^{1}_{(\cosh-1)_{*}}(M_{1})$ with 

$$\int_{\reals^{n-1}}M_{n-1}(\underline{x}) d \underline{x}\int_{\reals}\Phi_{*}(F_{\underline{x}}(y))M_{1}(y)d y=\int_{\reals^{n}}\Phi_{*}(f(x))\,M(x)d x$$  and 

$$\int_{\reals^{n-1}}M(\underline{x})d \underline{x}\int_{\reals}M_{1}(y)\Phi_{*}(F^{'}_{\underline{x}}(y)) d y =\int_{\reals^{n}}M(x)\Phi_{*}(\partial_j f(x))\,d x$$
where $F'$ denotes the distributional derivative of $F=F(y)$. In particular, if there exists $C > 0$ such that

\begin{equation}\label{eq:1d}\int_{\reals} \Phi_{*}(yF(y))M_{1}(y)d y \leq C \int_{\reals} \left(\Phi_{*}(F(y)) + \Phi_{*}(F'(y))\right)M_{1}(y)d y \qquad \forall F \in W^{1}_{(\cosh-1)_{*}}(M_{1})\end{equation}
we get the desired result.

Let us then prove \eqref{eq:1d} and fix $F \in W^{1}_{(\cosh-1)_{*}}(M_{1})$. From $ \Phi_*(y) = \displaystyle\int_0^{\absoluteval y} \arsinh u \ du$ together with the evenness of $\arsinh$  we obtain 

$$
  \Phi_*( yF(y))=\int_0^{\absoluteval {F(y)}} \absoluteval{y} \arsinh (\absoluteval{y}v) \ dv = \int_0^{\absoluteval {F(y)}} y\arsinh (y v) \ dv.
$$
Write for simplicity

$$G(y):=\int_0^{\absoluteval {F(y)}} \arsinh (y v) \ dv$$
one has

\begin{equation*}\begin{split}
  \int_{\reals} \Phi_*(y F(y))M_{1}(y)\ dy &= \int_{\reals}  y M_{1}(y)G(y)\ dy = - \int_{\reals} G(y)  M'_{1}(y)\ dy \\
  &= \int_{ \reals}     M_{1}(y)G'(y)\ dy.
\end{split}\end{equation*}
Now, the derivative of $G$ exists because of the assumption   $F \in W^1_{(\cosh-1)_*}(M_{1})$ (that is, $\absoluteval F \in W^1_{(\cosh-1)_*}(M_{1})$ and its derivative is given by the derivation of a composite function) and it is computed as

$$
 G'(y) = \arsinh(y \absoluteval{F(y)})  \dfrac{d}{d y} \absoluteval{F(y)} + \int_0^{\absoluteval {F(y)}} \frac{v}{\sqrt{1+y^2 v^2}} \ dv.$$
 Using Young's inequality with $\Phi=\cosh-1$ and $\Phi_{*}=(\cosh-1)_{*}$ we get 

$$
G'(y) \le \Phi\bigg(\arsinh(y \absoluteval{F(y)})\bigg) + \Phi_*\bigg(  \dfrac{d}{d y} \absoluteval{F(y)}\bigg) + \frac{\sqrt{1+y^2 \absoluteval{F(y)}^2}-1}{y^2}.
$$
All the terms in the right-hand side of the above inequality are integrable with respect to  the measure $M_{1}(y)\ dy$ over $\reals$. Indeed the first term is bounded as 

\begin{equation*}
  \Phi\left(\arsinh(y \absoluteval{F(y)})\right) = \sqrt{1+y^2 \absoluteval{F(y)}^2} \le \sqrt{2}\left(1 \vee \absoluteval{y F(y)}\right)
\end{equation*}
and $y \mapsto yF(y) \in L^{1}(\reals, M_{1}(y)dy)$. 
The second term is integrable by assumption. The only concern is then the last term. For any $r > 0$, 

$$\int_{|y| > r}M_{1}(y) \frac{\sqrt{1+y^2 \absoluteval{F(y)}^2}-1}{y^2}d y \leq \frac{1}{r} \int_{\reals} \absoluteval{F(y)}M_{1}(y) d y \leq \frac{1}{r}\int_{\reals} \Phi_{*}(F(y))M_{1}(y)d y$$
while, 

$$\int_{|y| < r}M_{1}(y) \frac{\sqrt{1+y^2 \absoluteval{F(y)}^2}-1}{y^2}d y \leq \int_{-r}^{r}M_{1}(y)\,\dfrac{\absoluteval{F(y)}}{\absoluteval{y}}d y \leq \frac{1}{\sqrt{2\pi}} \int_{-r}^{r}\dfrac{\absoluteval{F(y)}}{\absoluteval{y}}\,d y.$$
Now, splitting the integral into the two integrals $\int_{0}^{r}$ and $\int_{-r}^{0}$, one can use the one-dimensional Hardy inequality in Orlicz-Sobolev space \cite{Cianchi:1999} to get that there exists $C >0$ such that

\begin{equation*}
 \int_{-r}^{r}\dfrac{\absoluteval{F(y)}}{\absoluteval{y}}\,d y \leq C\int_{-r}^{r}\Phi_{*}(F'(y))\,d y.\end{equation*}
 This achieves to prove \eqref{eq:1d}.
\item Recall that \eqref{eq:stein} holds for any $f \in C_0^\infty(\reals^{n})$ and any $g \in \Lexp M$. Since $\Phi_{*}=(\cosh-1)_{*}$ enjoys the $\Delta_2$-condition, it is a well-known fact that $C_{0}^{\infty}(\reals^{n})$ is dense in $W_{(\cosh-1)_*}^{1}(M)$ (for the norm $\|\cdot\|_{W^{1}_{(\cosh-1)_{*}(M)}}$). Therefore, approximating any $f \in  W^{1}_{(\cosh-1)_{*}}(M)$ by a sequence of $C_{0}^{\infty}(\reals^{n})$ functions, we deduce the result from point 1.  
\end{enumerate}
\end{proof}

\begin{remark}
In the second Item of the above Proposition, notice that \emph{a priori} $X_{j} f$ belongs to $\LlogL M$ but not to $W^{1}_{(\cosh-1)_{*}}(M)$, $j=1,\ldots,n$. For this to be true, one would require that $\partial_{k} f \in W^{1}_{(\cosh-1)_{*}}(M)$ for any $k=1,\ldots,n$. 
\end{remark}

%
$$\|X_{j}f\|_{\LlogL M} \leq \delta + C_{\delta}\,\|f\|_{W^{1}_{(\cosh-1)_{*}(M)}} \qquad \forall f \in W^{1}_{(\cosh-1)_{*}}(M).$$


\subsection{Stein and Laplace operators.}

Following the language of \cite[Chapter V]{MR1335234}, Item 3 of the above Proposition can be reformulated saying that 

$$\scalarat M f{\partial_j g}=\scalarat M {\bdelta_{j}f}{g} \qquad \forall j=1,\ldots,n$$
where 

$$\bdelta_{j}f=X_{j}f-\partial_{j} f.$$
This allows to define the \emph{Stein operator} $\bdelta$ on $\LlogL M$ as

$$\bdelta\::\: f \in \mathrm{Dom}(\bdelta) \subset \LlogL M \longmapsto \bdelta f =(\bdelta_{j}f)_{j=1,\ldots,n} \in \left(\LlogL M\right)^{n}.$$
where the the domain $\mathrm{Dom}(\bdelta)$ of $\bdelta$ is exactly $W_{(\cosh-1)_{*}}^{1}(M)$ according to point 2. of the above Proposition. Notice that, since $\Phi_{*}$ enjoys the $\Delta_{2}$-condition, $\mathrm{Dom}(\bdelta)$ is dense in $\LlogL M$. One deduces then easily that $\bdelta$ is a closed and densely defined operator in $\LlogL M$. 
 
One sees that  

\begin{multline}\label{eq:bdelta}
\scalarat M f{\mathrm{div\,} \mathbf g}:=\sum_{j=1}^{n}\scalarat M f{\partial_j g_{j}}=\sum_{j=1}^{n}\scalarat M {\bdelta_{j}f}{g_{j}}=:\scalarat M {\bdelta f}{\mathbf g} \\\qquad \forall f \in W^{1}_{\cosh-1}(M)\,;\,\mathbf g=(g_{j})_{j=1,\ldots,n} \in \left(W^{1}_{(\cosh-1)_{*}}(M)\right)^{n}\end{multline}
where $\mathrm{div\,} \mathbf g=\sum_{j=1}^{n}\partial_{j}g_{j}$ is the divergence of $\mathbf g$. This allows to define the adjoint operator $\bdelta^{*}$ as follows, see \cite{MR2759829}

$$\bdelta^{*}\::\:\mathrm{Dom}(\bdelta^{*}) \subset \left(\Lexp M\right)^{n} \to \Lexp M$$
with

\begin{multline*}
\mathrm{Dom}(\bdelta^{*}) = \\
\setof{\mathbf g=(g_{j})_{j=1,\ldots,n} \in \left(\Lexp M\right)^{n}}{(\exists c > 0) (\forall f \in \mathrm{Dom}(\bdelta)) \  \absoluteval{\scalarat M {\mathbf g} {\bdelta f}} \le c\normat {\LlogL M} f}\end{multline*}
and

$$\scalarat M {\mathbf g} {\bdelta f}=\scalarat M {\bdelta^{*} \mathbf g} f \qquad \forall \mathbf g \in \mathrm{Dom}(\bdelta^{*})\;,\;\forall f \in \mathrm{Dom}(\bdelta).$$
One sees from \eqref{eq:bdelta} that

$$\left(W^{1}_{\cosh-1}(M)\right)^{n} \subset \mathrm{Dom}(\bdelta^{*}) \qquad \text{ and } \quad \bdelta^{*}\mathbf g=\bnabla \cdot\mathbf g \qquad \forall \mathbf g \in \left(W^{1}_{\cosh-1}(M)\right)^{n}.$$
%
$$\left|\sum_{j=1}^{n}\scalarat M {g_{j}} {\bdelta_{j} f}\right| \leq c \normat {V_{*}} f \qquad \forall f \in W^{1}_{(\cosh-1)_{*}}(M).$$
%
$$\left|\scalarat M {g_{j}} {X_{j} f}\right| \leq 2\normat {\Lexp M} {g_{j}} \,\normat {\LlogL M} {X_{j}f}$$ and one sees that there exists $C >0$ such that, for any $j=1,\ldots,n$ and any $f \in W^{1}_{(\cosh-1)_{*}}(M)$, it holds

$$\left|\scalarat M {g_{j}} {\partial_{j}  f}\right| \leq C \normat {V_{*}} f.$$
%
$$\left|\scalarat M {\partial_{j}g_{j}} f\right| \leq C\normat {V_{*}} f \qquad \forall f \in W^{1}_{(\cosh-1)*}(M)$$
\begin{remark}
Notice that, since $\Phi=\cosh-1$ does not satisfies the $\Delta_{2}$-condition,  it is not clear whether $\mathrm{Dom}(\bdelta^{*})$ is dense in $\left(\Lexp M\right)^{n}$ or not. However, from the general theory of adjoint operators and since $\bdelta$ is a closed densely defined operator in  $\LlogL M$, the domain $\mathrm{Dom}(\bdelta^{*})$ is dense in $\left(\Lexp M\right)^{n}$ endowed with the weak-$\star$ topology. Moreover, $\bdelta^{*}$ is a closed operator from $\left(\Lexp M
\right)^{n}$ to $\Lexp M$ (see \cite[Chapter 2]{MR2759829} for details).\end{remark}

We also define the gradient operator 

$$\bnabla\::\:\mathrm{Dom}(\bnabla) \subset \Lexp M \to \left(\Lexp M\right)^{n}$$
by 

$$\mathrm{Dom}(\bnabla)=W^{1}_{\cosh-1}(M) \qquad \text{ and  }  \qquad \bnabla f=(\partial_{j}f)_{j=1,\ldots,n}  \qquad f \in W^{1}_{\cosh-1}(M).$$ One sees that, if $f \in \mathrm{Dom}(\bnabla^{2})$, i.e. if $f \in \in W^{1}_{\cosh-1}(M)$ is such that $\bnabla f \in \left(\in W^{1}_{\cosh-1}(M)\right)^{n}$, then $\bnabla f \in \mathrm{Dom}(\bdelta^{*})$ and 

$$\bdelta^{*} \bnabla f=\sum_{j=1}^{n} \partial_{jj}^{2}f=\bDelta f.$$
corresponds to the Laplace operator. 

From now, set $W^{-1}_{\cosh-1}(M)$ as  the dual space of $W_{(\cosh-1)_{*}}(M)$:

$$W^{-1}_{\cosh-1}(M):=\left(W^{1}_{(\cosh-1)_{*}}(M)\right)^{*}$$
i.e. the set of all continuous linear forms continuous  $F\::\:W^{1}_{(\cosh-1)_{*}}\to \reals$ and let $\scalarat {-1,1}{\cdot}{\cdot}$ denotes the duality pairing: 

$$\scalarat {-1,1} F u =F(u) \qquad \forall u \in W^{1}_{(\cosh-1)_{*}}(M)\,;\,F \in W^{-1}_{\cosh-1}(M).$$
For any $f \in \mathrm{Dom}(\bnabla)$, let

$$L_{f}\::\: g \in W^{1}_{(\cosh-1)_{*}}(M) \longmapsto \scalarat M {\bnabla f} {\bdelta g} \in \reals.$$
One easily checks that $L_{f}$ is continuous and defines therefore an element of  $W^{-1}_{\cosh-1}(M).$  Clearly, the operator $f \in \mathrm{Dom}(\bnabla)=W^{1}_{\cosh-1}(M) \mapsto L_{f} \in W^{-1}_{\cosh-1}(M)$ is also linear. Using the identification $\bDelta=\bdelta^{*}\bnabla$, we denote it by
 
\begin{equation*}\begin{split}
\Delta\;:\:W^{1}_{\cosh-1}(M) &\longrightarrow W^{-1}_{\cosh-1}(M)\\ f &\longmapsto \Delta f:=L_{f}\end{split}\end{equation*}
defined as 

$$\scalarat{-1,1} {\Delta f}  u =\scalarat M {\bnabla f}{\bdelta u}, \qquad f \in W^{1}_{\cosh-1}(M)\,,\,u \in W^{1}_{(\cosh-1)_{*}}(M).$$
Notice that, with this definition, $\Delta f$ is an element of $W^{-1}_{\cosh -1}(M)$ whereas, from the above observation, $\bDelta f \in \Lexp M.$ Also the domains of both operators are different. Of course, if $f \in \mathrm{Dom}(\bnabla^{2})$ then $\Delta f$ is actually an element of $\Lexp M$ and it coincides with $\bDelta f,$ namely, in such a case

$$\scalarat{-1,1} {\Delta f} u = \scalarat M {\bDelta f} u \qquad \forall u \in W^{1}_{(\cosh-1)_{*}}(M).$$
\begin{remark}
Given $f \in W^{1}_{\cosh-1}(M)$ and $u \in C_0^\infty(\reals^n)$ one has

$$\scalarat{-1,1} {\Delta f}  u =\scalarat M {\bnabla f}{\bdelta u}=\sum_{j=1}^{n}\scalarat M {\partial_{j} f}{\delta_{j} u}$$
but, since $\delta_{j} u \in W^{1}_{(\cosh-1)_{*}}(M)$, one can use \eqref{eq:stein} again to get

$$\scalarat{-1,1} {\Delta f}  u =\sum_{j=1}^{n}\scalarat M   f {\delta_{j}^{2} u}.$$
One readily computes, for any $j=1,\ldots,n,$ 
$\delta_{j}^{2} u=(x_{j}^{2}-1)u -2 x_{j}\partial_{j} u +\partial_{jj}^{2} u$,
so that

$$\scalarat{-1,1} {\Delta f} u = \scalarat M f {\left(|{\bm X}|^{2}-1\right) u - 2 {\bm X} \cdot \bnabla u +\bDelta u} \qquad \forall u \in C_{0}^{\infty}(\reals^{n}).$$
\end{remark}

\begin{remark}
Since the constant function $1 \in W^{1}_{(\cosh-1)_{*}}(M)$ one checks easily that

$$\scalarat{-1,1} {\Delta f} 1 = \scalarat M {\bm X} {\bnabla f}   = \expectat M {{\bm X} \cdot \bnabla f} \quad \forall f \in W^{1}_{\cosh-1}(M).$$
where we notice that ${\bm X} \cdot \bnabla f \in L^{1}(M)$ for any $f \in W^{1}_{\cosh-1}(M).$ Actually, since ${\bm X} \in \left(W^{1}_{(\cosh-1)_{*}}(M)\right)^{n}$ one can use \eqref{eq:stein} again to get 

$$\scalarat M {\bm X} {\bnabla f} = \sum_{j=1}^{n}\scalarat M {\delta_{j} x_{j}} f =\sum_{j=1}^{n} \scalarat M {x_{j}^{2}-1} f$$
i.e.

$$\scalarat{-1,1} {\Delta f} 1=\scalarat M {|{\bm X}|^{2}-n} f=\int_{\reals^{n}} f(x)\bDelta M(x) d x\quad \forall f \in W^{1}_{\cosh-1}(M).$$
Of course, if $f \in W^{1}_{\cosh-1}(M)$ and $\bnabla f \in \left(W^{1}_{\cosh-1}(M)\right)^{n}$ one actually gets $\Delta f=\bDelta f \in \Lexp M$ and 

$$\expectat M {\Delta f} =\expectat M {{\bm X} \cdot \bnabla f} .$$
\end{remark}

For technical purposes, we finally state the following Lemma
\begin{lemma}
Given $w_{1},w_{2} \in W^{1}_{\cosh-1}(M)$, $v \in \Sspace M \cap W^{1}_{\cosh-1}(M)$ and $g=\euler_{M}(v)$ one has 
\begin{equation}\label{eq:w1w2}\scalarat{-1,1} {\Delta w_{2}} {w_{1}\euler^{v-K_{M}(v)}}=-\expectat g {\bnabla w_{1}\cdot\bnabla w_{2}} + \expectat g {w_{1}\left({\bm X}-\bnabla v\right)\cdot \bnabla w_{2}}.\end{equation}\end{lemma}
\begin{proof}
For simplicity, set $h=w_{1}\euler^{v-K_{M}(v)}.$ One knows from Proposition \ref{prop:feuler} that $h \in W^{1}_{(\cosh-1)_{*}}(M)$ so that, by definition

$$\scalarat{-1,1} {\Delta w_{2}} h =\scalarat M {\bnabla w_{2}} {\bdelta h}.$$
Moreover, one checks easily that

$$\bdelta h={\bm X} h-\bnabla h= {\bm X} h - \bnabla w_{1} \euler^{v-K_{M}(v)} -h \bnabla v$$
so that

$$\scalarat{-1,1} {\Delta w_{2}} h =\scalarat M {\bnabla w_{2}} {\left({\bm X}-\bnabla v\right) h} - \scalarat M {\bnabla w_{2}} {\bnabla w_{1} \euler^{v-K_{M}(v)}} $$
which gives the result.
\end{proof}

\subsection{Exponential family based on Orlicz-Sobolev spaces with Gaussian weights}

 If we restrict the exponential family $\maxexp M$ to $M$-centered random variable $W^{1}_{\cosh-1}(M)$, that is in 

 \begin{equation*}
   W_M = W^{1}_{\cosh-1}(M) \cap \Bspace M=\setof{U \in W^{1}_{\cosh-1}(M)}{\expectat M {U}=0} \ ,
 \end{equation*}
 we obtain the following non parametric exponential family

 \begin{equation*}
   \maxexpat {1} M = \setof{\euler^{U - K_M(U)} \cdot M}{U \in W^{1}_{\cosh-1}(M) \cap \sdomain M} \ .
 \end{equation*}
Because of the continuous embedding $W^{1}_{\cosh-1}(M) \hookrightarrow L^{\cosh-1}(M)$ the set $W^{1}_{\cosh-1}(M) \cap \sdomain M$ is open in $W_M$ and the cumulant functional $K_M : W^{1}_{\cosh-1}(M) \cap \sdomain M \to \reals$ is convex and differentiable. 

In a similar way, we can define

\begin{equation*}
  \prescript{*}{}W_M = W^{1}_{(\cosh-1)_{*}}(M) \cap \preBspace M=\setof{f \in W^{1}_{(\cosh-1)_{*}}(M)}{\expectat M {f}=0}
\end{equation*}
so that for each $f \in \maxexpat 1 M$ we have $\frac f M -1 \in \prescript{*}{}W_M$, see Remark  \ref{rem:wheref}. 

Every feature of the exponential manifold carries over to this case. In particular, we can define the spaces

\begin{equation*}
  W_f = W^{1}_{\cosh-1}(M) \cap \Bspace M=\setof{U \in W^{1}_{\cosh-1}(M)}{\expectat f U = 0}, \quad f \in \maxexpat 1 M \ ,
\end{equation*}
to be models for the tangent spaces of $\maxexpat 1 M$. Note that the transport acts on these spaces

\begin{equation*}
  \transport f g \colon W_f \ni U \mapsto U - \expectat g U \in W_g \ ,
\end{equation*}
so that we can define the tangent bundle to be

\begin{equation*}
  T \maxexpat 1 M = \setof{(g,V)}{ g \in \maxexpat 1 M, V \in W_f}
\end{equation*}
and take as charts the restrictions of the charts defined on $T\maxexp M$.

As a first example of application, note that the gradient of the BG-entropy 
$$\nabla H(f) = - \log f - H(f)$$
 is a vector field on $\maxexpat 1 M$, which implies the solvability in $\maxexpat 1 M$ of the gradient flow equation. Our concern here is to set up a framework for the study of evolution equation in $\maxexpat 1 M$. Following sections are devoted to discuss a special functional and its gradient.

\subsection{Hyv\"arinen divergence}
\label{sec:hyvarinen}

We begin with the following general properties of $W_M = \setof{U \in W^1_{\cosh-1}(M)}{\expectat M U = 0}$
 
\begin{proposition}
\label{prop:nabla-score}
Let $f, g \in \maxexpat 1 M$, with $f = \euler_M(u), g = \euler_M(v)$, $u,v \in W^1_{\cosh-1}(M) \cap \sdomain M$. The following hold

  \begin{enumerate}
  \item $\expectat M {\bnabla u} = \covat M u {\bm X}$.
  \item $\expectat g {\bm X - \bm \nabla v} = 0$.
  \item $\expectat g {\bm \nabla u} = \covat g u {\bm X - \bm \nabla v}$.
  \item $\expectat g {\bnabla u - \bnabla v} = \covat g {u-v} {\bm X - \bm \nabla v}$.
  \end{enumerate}
\end{proposition}

\begin{proof} In all the sequel, we set $F=\euler^{u-K_{M}(u)}=\frac{f}{M}$ and $G=\euler^{v-K_{M}(v)}=\frac{g}{M}.$ Recall from Proposition \ref{prop:feuler} that $F,G \in W^{1}_{(\cosh-1)_{*}}(M).$

  \begin{enumerate}
  \item One has $\expectat M {\bnabla u}=\scalarat M 1{\partial_j u}$ where $1 \in W^{1}_{(\cosh-1)_{*}}(M)$ is the constant function. Then, from  \eqref{eq:stein}, $\expectat M {\bnabla u}=\scalarat M {\bdelta 1} u=\covat M u {\bm X}$ since $\bdelta 1={\bm X}$.
  \item As above, one has
  
$$\expectat g {\bnabla v}=\int {\bnabla v}\euler_{M}(v) dx=\int {\bnabla v} G M\,d x$$
  and, from Proposition \ref{prop:feuler}, $\bnabla v G=\bnabla G$ so that
  
$$\expectat g {\bnabla v}=\expectat M {\bnabla G}.$$
  Recall that $G \in W^{1}_{(\cosh-1)_{*}}(M)$ and that ${\bnabla G}={\bm X}G-\bdelta G$ so that
  
$$\expectat g {\bnabla v}=\scalarat M {{\bm X}G} 1 - \scalarat M {\bdelta G} 1$$
  where $1 \in W^{1}_{\cosh-1}(M)$ is the constant function equal to 1. Applying again \eqref{eq:stein} we get $\scalarat M {\bdelta G} 1=\scalarat M G {\bnabla 1}=0$ and 
  
$$\expectat g {\bnabla v}=\scalarat M {{\bm X}G} 1=\expectat g {\bm X}.$$
\item Observe first that $\expectat g {\bnabla u}=\scalarat M G {\bnabla u} .$ Using again \eqref{eq:stein}, since $u \in W^{1}_{\cosh-1}(M)$ and $G \in W^{1}_{(\cosh-1)_{*}}(M)$ we have

$$\expectat g {\bnabla u}=\scalarat M {\bdelta G} u=\scalarat M {{\bm X} G-{\bnabla G}} u.$$
Since $GM=g$ and ${\bnabla G}={\bnabla v}G$ (see Proposition \ref{prop:feuler}) we get $\scalarat M {{\bm X} G-{\bnabla G}} u=\expectat g {\left({\bm X-\bnabla v}\right)u}$ which gives the result.
\item Arguing as above one sees that $\expectat g {\bnabla v}=\expectat g {\left({\bm X-\bnabla v}\right) v}$ and therefore the conclusion follows from the previous item.
  \end{enumerate}
\end{proof}

\begin{remark}\ 
\begin{enumerate}
\item
The Eq.s in Prop. \ref{prop:nabla-score} could be written without reference to the score (chart) mapping $s_M \colon f \mapsto u$ by writing $\bnabla u = \bnabla \logof{\frac{f}{M}}=\bnabla \log f+{\bm X}$ to get
  \begin{enumerate}
  \item $\expectat M {\bnabla \logof{\frac{f}{M}}} = \covat M {\logof{\frac{f}{M}}}{\bm X}$.
  \item $\expectat g {\bnabla \log g} = 0$.
  \item $\expectat g {\bnabla \logof{\frac{f}{M}}} = -\covat g {\logof{\frac{f}{M}}}{\bnabla \log g}$.
  \item $\expectat g {\bnabla\logof{\frac f g}} = -\covat g {\logof{\frac f g}} {\bnabla\logof{\frac g M}}$.
\end{enumerate}
However, we feel that explicit reference to the chart clarifies the geometric picture.
\item The mapping $f \mapsto \bnabla \log f$ is a vector field in $T\maxexpat 1 M$ with flow given by the translations:

  \begin{equation*}
    \derivby t \log f_t = \bnabla \log f_i, \quad f(\bm x,t) = f(\bm x + \bm 1 t).
  \end{equation*}
\item
The KL-divergence $(f,g) \mapsto \KL f g$ has expression $(u,v) \mapsto dK_M(v) [u-v] - K_M(u) + K_M(v)$ in the chart centered at $M$, with partial derivative with respect to $v$ in the direction $w$ given by $\covat g {u-v} {w}$. If the direction is $w = \bm X - \bnabla v = -\bnabla\log g$, we have that $\expectat g {\bnabla v - \bnabla u}$ is the derivative of the KL-divergence along the vector field of translations.
\end{enumerate}
\end{remark} 

We introduce here the \emph{Hyv\"arinen divergence} between two elements of $\maxexpat 1 M$:

\begin{definition}[Hyv\"arinen divergence] 
For each $f,g \in \maxexpat 1 M$ the \emph{Hyv\"arinen divergence} is the quantity

\begin{equation*}
  \KH g f = \expectat g {\absoluteval{\bnabla \log f - \bnabla \log g}^2}. 
\end{equation*}
The expression in the chart centered at $M$ is

\begin{equation*}
  \operatorname{DH}_M(v \| u):=\KH {\euler_{M}(v)} {\euler_{M}(u)}= \expectat M {\absoluteval{\bnabla u - \bnabla v}^2 \euler^{v - K_M(v)}}, 
\end{equation*}
where $f = \euler_M(u)$, $g = \euler_M(v)$.
\end{definition}
\begin{remark}
  \begin{enumerate}
  \item The mapping $\maxexpat 1 M \ni f \mapsto \bnabla \log f = f^{-1} \bnabla f$ is, in statistical terms, an \emph{estimating function} or a \emph{pivot}, because $\expectat f {\bnabla \log f} = 0$, $f \in \maxexpat 1 M$. This means, it is a random variable whose value is zero in the mean if $f$ is correct. If $g$ is correct, then the expected value is $\expectat g {\bnabla \log f} = \expectat g {\bnabla \logof{\frac f g}} = -\covat g {\logof{\frac f g}}{\bnabla\logof{\frac g M}}$. The second moment of $\bnabla \logof{\frac f g}$ was used by Hyv\"arinen as a measure of deviation from $f$ to $g$, \cite{MR2249836}
\item Hyv\"arinen work has been used to discuss \emph{proper scoring rules} in \cite{MR3014317}.
\item The same notion is known in Physics under the name of \emph{relative Fisher information} e.g., see \cite{MR2409050}. 
  \end{enumerate}
\end{remark}

In the following we denote the gradient of a function defined on the exponential manifold $\maxexpat 1 M$, which is a random variable,  by $\partial $.

\begin{proposition}\  
  \begin{enumerate}
  \item The Hyv\"arinen divergence is finite and infinitely differentiable everywhere in both variables.
  \item $\partial (f \mapsto \KH g f) = -2 \bnabla \log g \cdot \bnabla \log {\frac f g} - 2\Delta  \log {\frac f g}$
  \item $\partial (g \mapsto \KH f g) = 2\bnabla \log g\cdot \bnabla \log {\frac f g} +2\Delta  \log {\frac f g} + \KH f g$.
  \end{enumerate}
\end{proposition}

\begin{proof}\ 
  \begin{enumerate}
  \item For each $w \in V$ the gradient $\bnabla w$ is in $(L^\Phi(g))^n = (L^\Phi(M))^n$ for all $g \in \maxexpat 1 M$, hence it is $g$-square integrable for all $g \in \maxexpat 1 M$. Moreover, the squared norm function $\Sspace M \times (B_M)^n \ni (v,\bm w) \mapsto \expectat {\euler_M(v)} {\absoluteval{\bm w}^2}$ is $\infty$-differentiable because it is the moment functional,

    \begin{equation*}
      \expectat g {\absoluteval{\bm w}^2} = \sum _{j=1}^n \left(d^2K_M(v)[w_j, w_{j}] + (dK_M(v)[w_j])^2\right).
    \end{equation*}
We can compose this function with the linear function

\begin{equation*}
  V_M\cap \Sspace M \times V_M\ni (v,u) \mapsto (v,\bm w) = (v,\bnabla (u-v)) \in \Sspace M \times B_M^n.
\end{equation*}

  \item Let $g=\euler_{M}(v), f=\euler_{M}(u)$, $u, v \in \Sspace M \cap V$ be given. For any $w \in V$, we compute first the directional derivative:

    \begin{equation*}
      d \left(u \mapsto \operatorname{DH}_M(v \Vert u)\right)[w] =   2 \expectat M {\bnabla w \cdot (\bnabla u - \bnabla v) \PSexp M v} = 2 \expectat {g} {\bnabla w \cdot (\bnabla u - \bnabla v)}
    \end{equation*}
where we notice that all the terms are well defined whenever $u,v \in \Sspace M \cap V$, $w \in V.$ Using now \eqref{eq:w1w2} with $w_{2}=u-v$ and $w_{1}=w$ we get that

\begin{multline*}
d \left(u \mapsto \operatorname{DH}_M(v \Vert u)\right)[w]=2\expectat{g} {w\left({\bm X}-\bnabla v\right)\cdot \bnabla (u-v)} - 2\scalarat{-1,1} {\Delta (u-v)} {w\euler^{v-K_{M}(v)}}\\
=2 \expectat M {w\euler^{v-K_{M}(v)}\left({\bm X}-\bnabla v\right)\cdot \bnabla (u-v)} -2\scalarat{-1,1} {\Delta (u-v)} {w\euler^{v-K_{M}(v)}}.
\end{multline*}
Since this is true for any $w \in V$ we get

$$d\left(u \mapsto \operatorname{DH}_M(v \Vert u)\right)=2\left({\bm X}-\bnabla v\right)\cdot \bnabla (u-v) -2\Delta (u-v)$$
where of course, $\Delta (u-v)$ is meant in $W^{-1}_{\cosh-1}(M)$ (notice that $w \in W^{1}_{\cosh-1}(M) \subset W^{1}_{(\cosh-1)_{*}}(M)$. 
The formula for the partial gradient in absolute variables follows from

\begin{equation*}
  (\bm X - \bnabla v)\cdot \bnabla (u-v)= - \bnabla \log g\cdot \bnabla \log {\frac f g}\end{equation*}
for, with a slight abuse of notations, we identify $\Delta \log {\frac f g}$ to $\Delta(u-v).$
\item As above, let $g=\euler_{M}(v), f=\euler_{M}(u)$, $u, v \in \Sspace M \cap V$ be given.  For any $w \in V$, we compute first the directional derivative. One gets now

     \begin{multline*} d \left(v \mapsto \operatorname{DH}_M(u \Vert v)\right)[w] = - 2 \expectat M {\bnabla w \cdot (\bnabla u - \bnabla v) \PSexp M v} \\ + \expectat M {\left(w - dK_M(v) [w]\right) \absoluteval{\bnabla u - \bnabla v}^2 \PSexp M v}\end{multline*}
     
One recognizes in the first term $- d\left(u \mapsto \operatorname{DH}_{M}( v \Vert u)\right)[w]$ while the second term is given by

\begin{multline*} \expectat M {\left(w - dK_M(v) [w]\right) \absoluteval{\bnabla u - \bnabla v}^2 \PSexp M v}=\expectat g {\left(w - \expectat g w\right) \absoluteval{\bnabla u - \bnabla v}^2} = \\  \covat g {w}{\absoluteval{\bnabla u - \bnabla v}^2}.
    \end{multline*}
As in the previous item, this gives the result.
\end{enumerate}
\end{proof}

As well-documented, the Hyvar\"inen divergence is a powerful tool for the study of general diffusion processes. We have just shown that the Information Geometry formalism and the exponential manifold approach are robust enough to allow for a generalization in  Orlicz-Sobolev spaces. We believe then that, as Boltzmann equation can be studied through the exponential manifold formalism in $\Lexp M$, general diffusion processes can be investigated in $W^{1}_{\cosh-1}(M)$ with the formalism discussed in the present section. This is a plan for future work.

\section{Conclusions and Discussion}
\label{sec:conclusion}

We have shown that well known geometric feature of problems in Statistical Physics can be turned into precise formal results via a careful consideration of the relevant functional analysis. 

In particular the notion of flow in a Banach manifold modeled on Orlicz spaces can be used to clarify arguments based on the evolution of the classical Boltzmann-Gibbs entropy in the vector field associated to the Boltzmann equation. 

In the last section we have shown how to construct a similar theory in the case the generalized entropy under consideration is the so-called Fisher functional or Hyvar\"inen divergence. Such a generalised entropy is particularly well-suited for the study of general diffusion problems and the results presented in Section 6 can be seen as the first outcome of an ongoing joint research program.

\bibliographystyle{amsplain}


\begin{thebibliography}{10}

\bibitem{MR960687}
R.~Abraham, J.~E. Marsden, and T.~Ratiu, \emph{Manifolds, tensor analysis, and
  applications}, second ed., Applied Mathematical Sciences, vol.~75,
  Springer-Verlag, New York, 1988. \MR{960687 (89f:58001)}

\bibitem{MR2424078}
Robert~A. Adams and John J.~F. Fournier, \emph{Sobolev spaces}, second ed.,
  Pure and Applied Mathematics (Amsterdam), vol. 140, Elsevier/Academic Press,
  Amsterdam, 2003. \MR{2424078 (2009e:46025)}

\bibitem{MR932246}
S.-I. Amari, O.~E. Barndorff-Nielsen, R.~E. Kass, S.~L. Lauritzen, and C.~R.
  Rao, \emph{Differential geometry in statistical inference}, Institute of
  Mathematical Statistics Lecture Notes---Monograph Series, 10, Institute of
  Mathematical Statistics, Hayward, CA, 1987. \MR{932246 (90d:62002)}

\bibitem{MR1800071}
Shun-ichi Amari and Hiroshi Nagaoka, \emph{Methods of information geometry},
  Translations of Mathematical Monographs, vol. 191, American Mathematical
  Society, Providence, RI; Oxford University Press, Oxford, 2000, Translated
  from the 1993 Japanese original by Daishi Harada. \MR{1800071 (2001j:62023)}

\bibitem{MR2401600}
Luigi Ambrosio, Nicola Gigli, and Giuseppe Savar{\'e}, \emph{Gradient flows in
  metric spaces and in the space of probability measures}, second ed., Lectures
  in Mathematics ETH Z\"urich, Birkh\"auser Verlag, Basel, 2008. \MR{2401600
  (2009h:49002)}

\bibitem{MR1066204}
J{\"u}rgen Appell and Petr~P. Zabrejko, \emph{Nonlinear superposition
  operators}, Cambridge Tracts in Mathematics, vol.~95, Cambridge University
  Press, Cambridge, 1990. \MR{1066204 (91k:47168)}

\bibitem{ay|jost|le|schwachhofer:2014}
Nihat Ay, J\"urgen Jost, H\^ong~V\^an L\^e, and Lorenz Schwachh\"ofer,
  \emph{Information geometry and sufficient statistics}, Probability Theory and
  Related Fields (2014), 38, OnLineFirst.

\bibitem{bourbaki:71variete}
Nicolas Bourbaki, \emph{Vari\'et\'es differentielles et analytiques. fascicule
  de r\'esultats / paragraphes 1 \`a 7}, \'El\'ements de math\'ematiques, no.
  XXXIII, Hermann, Paris, 1971.

\bibitem{MR0215617}
L.~M. Br{\`e}gman, \emph{A relaxation method of finding a common point of
  convex sets and its application to the solution of problems in convex
  programming}, \u Z. Vy\v cisl. Mat. i Mat. Fiz. \textbf{7} (1967), 620--631.
  \MR{0215617 (35 \#6457)}

\bibitem{MR2759829}
Haim Brezis, \emph{Functional analysis, {S}obolev spaces and partial
  differential equations}, Universitext, Springer, New York, 2011. \MR{2759829
  (2012a:35002)}

\bibitem{MR882001}
Lawrence~D. Brown, \emph{Fundamentals of statistical exponential families with
  applications in statistical decision theory}, Institute of Mathematical
  Statistics Lecture Notes---Monograph Series, 9, Institute of Mathematical
  Statistics, Hayward, CA, 1986. \MR{882001 (88h:62018)}

\bibitem{cena:2002}
Alberto Cena, \emph{Geometric structures on the non-parametric statistical
  manifold}, Ph.D. thesis, Dottorato in Matematica, Universit\`a di Milano,
  2002.

\bibitem{MR2396032}
Alberto Cena and Giovanni Pistone, \emph{Exponential statistical manifold},
  Ann. Inst. Statist. Math. \textbf{59} (2007), no.~1, 27--56. \MR{2396032
  (2009b:62011)}

\bibitem{MR1313028}
Carlo Cercignani, \emph{The {B}oltzmann equation and its applications}, Applied
  Mathematical Sciences, vol.~67, Springer-Verlag, New York, 1988. \MR{1313028
  (95i:82082)}

\bibitem{Cianchi:1999}
Andrea Cianchi, \emph{Some results in the theory of orlicz spaces and
  applications to variational problems}, Nonlinear Analysis, Function Spaces
  and Applications, Vol. 6. Czech Academy of Sciences, Mathematical Institute,
  Praha, 1999; Krbec, Miroslav and Kufner, Alois (eds.) (1999), Proceedings of
  the Spring School held in Prague, May 31-June 6, 1998.

\bibitem{MR0471125}
A.~P. Dawid, \emph{Further comments on: ``{S}ome comments on a paper by
  {B}radley {E}fron''\ ({A}nn. {S}tatist. {\bf 3} (1975), 1189--1242)}, Ann.
  Statist. \textbf{5} (1977), no.~6, 1249. \MR{0471125 (57 \#10863)}

\bibitem{MR1138207}
Manfredo~Perdig{\~a}o do~Carmo, \emph{Riemannian geometry}, Mathematics: Theory
  \& Applications, Birkh\"auser Boston, Inc., Boston, MA, 1992, Translated from
  the second Portuguese edition by Francis Flaherty. \MR{1138207 (92i:53001)}

\bibitem{MR0428531}
Bradley Efron, \emph{Defining the curvature of a statistical problem (with
  applications to second order efficiency)}, Ann. Statist. \textbf{3} (1975),
  no.~6, 1189--1242, With a discussion by C. R. Rao, Don A. Pierce, D. R. Cox,
  D. V. Lindley, Lucien LeCam, J. K. Ghosh, J. Pfanzagl, Niels Keiding, A. P.
  Dawid, Jim Reeds and with a reply by the author. \MR{0428531 (55 \#1552)}

\bibitem{MR1805840}
Paolo Gibilisco and Tommaso Isola, \emph{Connections on statistical manifolds
  of density operators by geometry of noncommutative {$L\sp p$}-spaces}, Infin.
  Dimens. Anal. Quantum Probab. Relat. Top. \textbf{2} (1999), no.~1, 169--178.
  \MR{1805840 (2003c:46085)}

\bibitem{MR1628177}
Paolo Gibilisco and Giovanni Pistone, \emph{Connections on non-parametric
  statistical manifolds by {O}rlicz space geometry}, Infin. Dimens. Anal.
  Quantum Probab. Relat. Top. \textbf{1} (1998), no.~2, 325--347. \MR{1628177}

\bibitem{MR2249836}
Aapo Hyv{\"a}rinen, \emph{Estimation of non-normalized statistical models by
  score matching}, J. Mach. Learn. Res. \textbf{6} (2005), 695--709.
  \MR{2249836}

\bibitem{imparato:thesis}
Daniele Imparato, \emph{Exponential models and {F}isher information. geometry
  and applications}, Ph.D. thesis, DIMAT Politecnico di Torino, 2008.

\bibitem{MR0039968}
S.~Kullback and R.~A. Leibler, \emph{On information and sufficiency}, Ann.
  Math. Statistics \textbf{22} (1951), 79--86. \MR{0039968 (12,623a)}

\bibitem{MR1335233}
Serge Lang, \emph{Differential and {R}iemannian manifolds}, third ed., Graduate
  Texts in Mathematics, vol. 160, Springer-Verlag, New York, 1995. \MR{1335233
  (96d:53001)}

\bibitem{MR1817225}
Elliott~H. Lieb and Michael Loss, \emph{Analysis}, second ed., Graduate Studies
  in Mathematics, vol.~14, American Mathematical Society, Providence, RI, 2001.
  \MR{1817225 (2001i:00001)}

\bibitem{MR3205750}
W.~Adam Majewski and Louis~E. Labuschagne, \emph{On applications of {O}rlicz
  spaces to statistical physics}, Ann. Henri Poincar\'e \textbf{15} (2014),
  no.~6, 1197--1221. \MR{3205750}

\bibitem{MR1335234}
Paul Malliavin, \emph{Integration and probability}, Graduate Texts in
  Mathematics, vol. 157, Springer-Verlag, New York, 1995, With the
  collaboration of H{\'e}l{\`e}ne Airault, Leslie Kay and G{\'e}rard Letac,
  Edited and translated from the French by Kay, With a foreword by Mark Pinsky.
  \MR{1335234 (97f:28001a)}

\bibitem{MR724434}
Julian Musielak, \emph{Orlicz spaces and modular spaces}, Lecture Notes in
  Mathematics, vol. 1034, Springer-Verlag, Berlin, 1983. \MR{724434
  (85m:46028)}

\bibitem{MR2948226}
Nigel~J. Newton, \emph{An infinite-dimensional statistical manifold modelled on
  {H}ilbert space}, J. Funct. Anal. \textbf{263} (2012), no.~6, 1661--1681.
  \MR{2948226}

\bibitem{MR3014317}
Matthew Parry, A.~Philip Dawid, and Steffen Lauritzen, \emph{Proper local
  scoring rules}, Ann. Statist. \textbf{40} (2012), no.~1, 561--592.
  \MR{3014317}

\bibitem{pistone:2009EPJB}
Giovanni Pistone, \emph{$\kappa$-exponential models from the geometrical
  viewpoint}, The European Physical Journal B Condensed Matter Physics
  \textbf{71} (2009), no.~1, 29--37.

\bibitem{pistone:2010SL}
\bysame, \emph{Algebraic varieties vs differentiable manifolds in statistical
  models}, Algebraic and geometric methods in statistics (Paolo Gibilisco, Eva
  Riccomagno, Maria~Piera Rogantin, and Henry~P. Wynn, eds.), Cambridge
  University Press, Cambridge, 2010, pp.~341--365.

\bibitem{MR3130268}
\bysame, \emph{Examples of the application of nonparametric information
  geometry to statistical physics}, Entropy \textbf{15} (2013), no.~10,
  4042--4065. \MR{3130268}

\bibitem{MR3126029}
\bysame, \emph{Nonparametric information geometry}, Geometric science of
  information, Lecture Notes in Comput. Sci., vol. 8085, Springer, Heidelberg,
  2013, pp.~5--36. \MR{3126029}

\bibitem{MR1704564}
Giovanni Pistone and Maria~Piera Rogantin, \emph{The exponential statistical
  manifold: mean parameters, orthogonality and space transformations},
  Bernoulli \textbf{5} (1999), no.~4, 721--760. \MR{1704564 (2000k:62005)}

\bibitem{MR1370295}
Giovanni Pistone and Carlo Sempi, \emph{An infinite-dimensional geometric
  structure on the space of all the probability measures equivalent to a given
  one}, Ann. Statist. \textbf{23} (1995), no.~5, 1543--1561. \MR{1370295
  (97j:62006)}

\bibitem{santacroce|siri|trivellato:2015}
Marina Santacroce, Paola Siri, and Barbara Trivellato, \emph{New results on
  mixture and exponential models by {O}rlicz spaces}, Bernoulli (2015), to
  appeat.

\bibitem{MR2293045}
Hirohiko Shima, \emph{The geometry of {H}essian structures}, World Scientific
  Publishing Co. Pte. Ltd., Hackensack, NJ, 2007. \MR{2293045 (2008f:53011)}

\bibitem{MR1700142}
G.~Toscani and C.~Villani, \emph{Sharp entropy dissipation bounds and explicit
  rate of trend to equilibrium for the spatially homogeneous {B}oltzmann
  equation}, Comm. Math. Phys. \textbf{203} (1999), no.~3, 667--706.
  \MR{1700142 (2000e:82039)}

\bibitem{MR2409050}
C.~Villani, \emph{Entropy production and convergence to equilibrium}, Entropy
  methods for the {B}oltzmann equation, Lecture Notes in Math., vol. 1916,
  Springer, Berlin, 2008, pp.~1--70. \MR{2409050}

\bibitem{MR1942465}
C{\'e}dric Villani, \emph{A review of mathematical topics in collisional
  kinetic theory}, Handbook of mathematical fluid dynamics, {V}ol. {I},
  North-Holland, Amsterdam, 2002, pp.~71--305. \MR{1942465 (2003k:82087)}

\end{thebibliography}

\def\cprime{$'$}
\providecommand{\bysame}{\leavevmode\hbox to3em{\hrulefill}\thinspace}
\providecommand{\MR}{\relax\ifhmode\unskip\space\fi MR }
\providecommand{\MRhref}[2]{%
  \href{http://www.ams.org/mathscinet-getitem?mr=#1}{#2}
}
\providecommand{\href}[2]{#2}

\end{document}